\documentclass[a4paper,10pt]{amsart}
\usepackage{amsmath,amssymb,amsthm,a4wide,graphicx}

\usepackage{caption}
\usepackage{amsthm}
\newtheorem{step}{Step}[section]
\newtheorem{theorem}{Theorem}[section]
\newtheorem{corollary}{Corollary}[section]
\newtheorem{lemma}{Lemma}[section]
\newtheorem{definition}{Definition}[section]
\newtheorem{remark}{Remark}[section]
\DeclareMathOperator*{\argmax}{argmax}
\DeclareMathOperator*{\argmin}{argmin}
\newcommand{\tr}{\operatorname{tr}}
\newcommand{\diag}{\operatorname{diag}}
\newcommand{\Id}{\operatorname{Id}}

\title[]{Stretching convex domains to capture \\ many lattice points}

\author[]{Nicholas F. Marshall}
\address{Department of Mathematics, Yale University, New Haven, CT 06511, USA}
\email{nicholas.marshall@yale.edu}

\begin{document}

\begin{abstract}
We consider an optimal stretching problem for strictly convex domains in
$\mathbb{R}^d$ that are symmetric with respect to each coordinate hyperplane,
where stretching refers to transformation by a diagonal matrix of determinant
$1$. Specifically, we prove that the stretched convex domain which captures the
most positive lattice points in the large volume limit is balanced: the
$(d-1)$-dimensional measures of the intersections of the domain with each coordinate hyperplane are equal. Our results extend those of Antunes \& Freitas, van den Berg, Bucur \& Gittins, Ariturk \& Laugesen, van den Berg \& Gittins, and Gittins \& Larson. The approach is motivated by the Fourier analysis techniques used to prove the classical $\#\{(i,j) \in \mathbb{Z}^2 : i^2 +j^2 \le r^2 \} =\pi r^2 + \mathcal{O}(r^{2/3})$ result for the Gauss circle problem.
\end{abstract}

\maketitle

\section{Introduction}

\subsection{Introduction}
In \cite{AntunesFreitas2012} Antunes \& Freitas introduced a new
type of lattice point problem; namely, among all ellipses that are symmetric
about the coordinate axes and of fixed area, which captures the most positive
lattice points? More precisely, for $r \ge 1$ what is:
$$
a(r) = \argmax_{a > 0}\, \# \left\{ (i,j) \in \mathbb{Z}^2_{>0} : \left(\frac{i}{a} \right)^2 +
\left( j a  \right)^2 \le r^2 \right \}.
$$
Determining $a(r)$ for fixed $r \ge 1$ is challenging, even computationally
\cite{AntunesFreitas2012}; however, Antunes \& Freitas were able
to prove that 
$$
\lim_{r \rightarrow \infty} a(r) = 1,
$$
i.e., the ellipse which captures the most positive lattice points for large
areas approaches a circle. Moreover, a rate of convergence of at least
$$
|a(r) - 1| = \mathcal{O}(r^{-1/6}) \quad \text{as } r \rightarrow \infty,
$$
was established in \cite{LaugesenLiu2016} (and implied
by Section 5 in \cite{AntunesFreitas2012}).
\begin{figure}[h!]
\centering
\includegraphics[width=.45\textwidth]{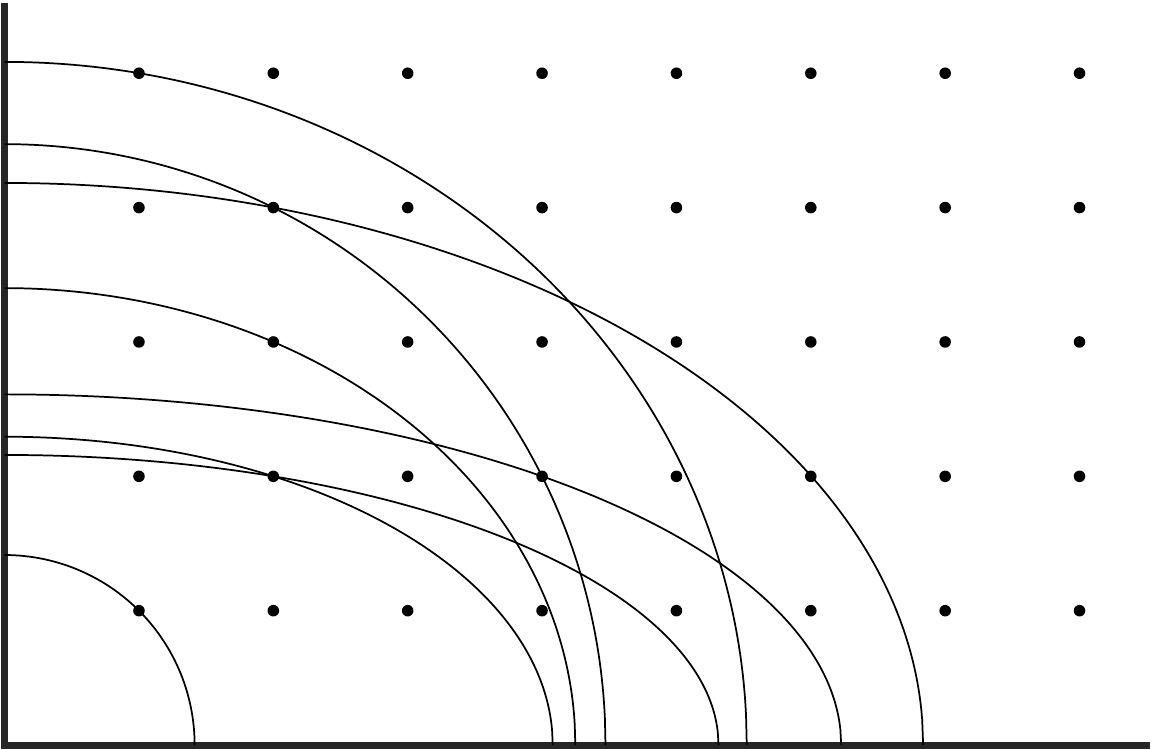}
\caption{Eight ellipses which each capture the maximum possible number of
positive lattice points for their given area.}
\end{figure}
This lattice point counting problem was originally motivated by a problem in
high frequency shape optimization. Suppose that $\Omega \subset \mathbb{R}^2$ is
an $a \times 1/a$ rectangular domain. Then the eigenvalues of the
Dirichlet Laplacian $-\Delta_\Omega^{\mathcal{D}}$ are of the form
$$
\sigma \left(-\Delta_\Omega^{\mathcal{D}} \right) = \left\{ \pi^2 \left( \left(
\frac{i}{a} \right)^2 + \left( a j \right)^2 \right) : i,j \in \mathbb{Z}_{>0} \right\}.
$$
Thus, there is a bijection between the Dirichlet eigenvalues less than $\pi^2
r^2$, and the positive lattice points in the ellipse $ (x/a)^2 + (a y)^2 \le
r^2$. Hence, the statement $\lim_{r \rightarrow \infty} a(r) = 1$ can also be
interpreted as the statement that the rectangle that minimizes the Dirichlet
eigenvalues in the high frequency limit approaches the square.

\subsection{Neumann eigenvalues}
A dual result for eigenvalues of the Neumann Laplacian
$-\Delta_\Omega^\mathcal{N}$ was established by van den Berg, Bucur \& Gittins
\cite{vandenBergBucurGittins2016}, where as before $\Omega \subset \mathbb{R}^2$
is an $a \times 1/a$ rectangular domain. The eigenvalues of
$-\Delta_\Omega^\mathcal{N}$ are of the form
$$
 \sigma \left(-\Delta_\Omega^\mathcal{N} \right) = \left \{ \pi^2 \left( \left(
\frac{i}{a} \right)^2 + \left( a j \right)^2 \right) : i,j \in \mathbb{Z}_{\ge
0} \right\},
$$
and hence, the Neumann eigenvalues less than $\pi^2 r^2$ are in bijection with
the nonnegative lattice points in the ellipse $ (x/a)^2 + (a y)^2 \le r^2$. In
terms of a lattice point problem, van den Berg, Bucur \& Gittins proved that if
$$
a(r) = \argmin_{a > 0}\, \# \left\{ (i,j) \in \mathbb{Z}^2_{\ge 0} :
\left(\frac{i}{a} \right)^2 + \left( j a  \right)^2 \le r^2 \right \},
$$
then $\lim_{r \rightarrow \infty} a(r) = 1$. That is, the ellipse which captures
the least nonnegative lattice points in the large area limit approaches the circle;
equivalently, the rectangle which maximizes the Neumann eigenvalues in the high
frequency limit approaches the square.  
\subsection{Higher dimensions}
Subsequently, the result for Dirichlet eigenvalues was generalized to
three dimensions by Gittins \& van den Berg \cite{vandenBergGittins2016}, and
recently to $d$-dimensions for both the Dirichlet and Neumann cases by Gittins
\& Larson \cite{GittinsLarson2017}.  Specifically, Gittins \& Larson show
that in $\mathbb{R}^d$ the cuboid of unit measure which minimizes the Dirichlet
eigenvalues in the high frequency limit approaches the cube, and similarly, that
the cuboid of unit measure which maximizes the Neumann Laplacian eigenvalues in
the high frequency limit approaches the cube. Both the Dirichlet and Neumann
cases have corresponding lattice point problems analogous to the $2$-dimensional
case. The Dirichlet case corresponds to the following lattice point problem.
Suppose $A = \diag(a_1,a_2,\ldots,a_d)$ is a positive diagonal matrix of
determinant $1$. For
$r \ge 1$, let 
$$ 
A(r) = \argmax_{A} \, \# \left\{ (i_1,i_2,\ldots,i_d) \in \mathbb{Z}^d_{>0} : \left( \frac{i_1}{a_1}
\right)^2 + \left( \frac{i_2}{a_2} \right)^2 + \cdots + \left( \frac{i_d}{a_d}
\right)^2 \le r^2 \right\}.
$$
Then the result of \cite{GittinsLarson2017} implies that $\lim_{r \rightarrow
\infty} \|A - \Id\|_\infty = 0$. Moreover, Gittins \& Larson
\cite{GittinsLarson2017} show the following error estimate holds for dimensions
$d \ge 2$: 
$$
 \left\| A - \Id \right\|_\infty = \mathcal{O} \left( r^{-\frac{d-1}{2(d+1)}}
\right) \quad \text{as } r \rightarrow \infty,
$$
and furthermore, they show a slightly improved error rate holds for
$d > 5$ by applying a result of G\"otze \cite{Gotze2004}.  In this paper, we
establish a similar rate of convergence for general convex domains. We note that
the Neumann case in $d$-dimensions has an analogous dual formulation to the
$2$-dimensional case, and similar error estimates were established in
\cite{GittinsLarson2017}.  We remark that results concerning the optimization of
eigenvalues of the Laplacian with perimeter constraints have also been
considered \cite{Antunes2016, BucurFreitas2013, vandenBerg2015}.  In particular,
Bucur \& Freitas \cite{BucurFreitas2013} proved that any sequence of minimizers
of the Dirichlet eigenvalues with a perimeter constraint converges to the unit
disk in the high frequency limit, and that among $k$-sided polygons, any
sequence of minimizers converges to the regular $k$-sided polygon in the high
frequency limit.

\subsection{Convex and concave curves}
Laugesen \& Liu \cite{LaugesenLiu2016} and  Ariturk \& Laugesen
\cite{AriturkLaugesen2017} generalized the two dimensional case of the above
lattice point counting problems to certain classes of convex and concave curves.
In particular, their results imply that among $p$-ellipses $ \left|x/a\right|^p
+ \left|a y \right|^p = r^p$ for $p \in (0,\infty) \setminus \{1\}$  the
$p$-ball captures the most positive lattice points. More generally, Laugesen \&
Liu and Ariturk \& Laugesen show that the curves which capture the most positive
integer lattice points in the large area limit are balanced: the distance from
the origin to their points of intersection with the $x$-axis and $y$-axis are
equal. 

\begin{figure}[h!]
\centering
\includegraphics[width=.23\textwidth]{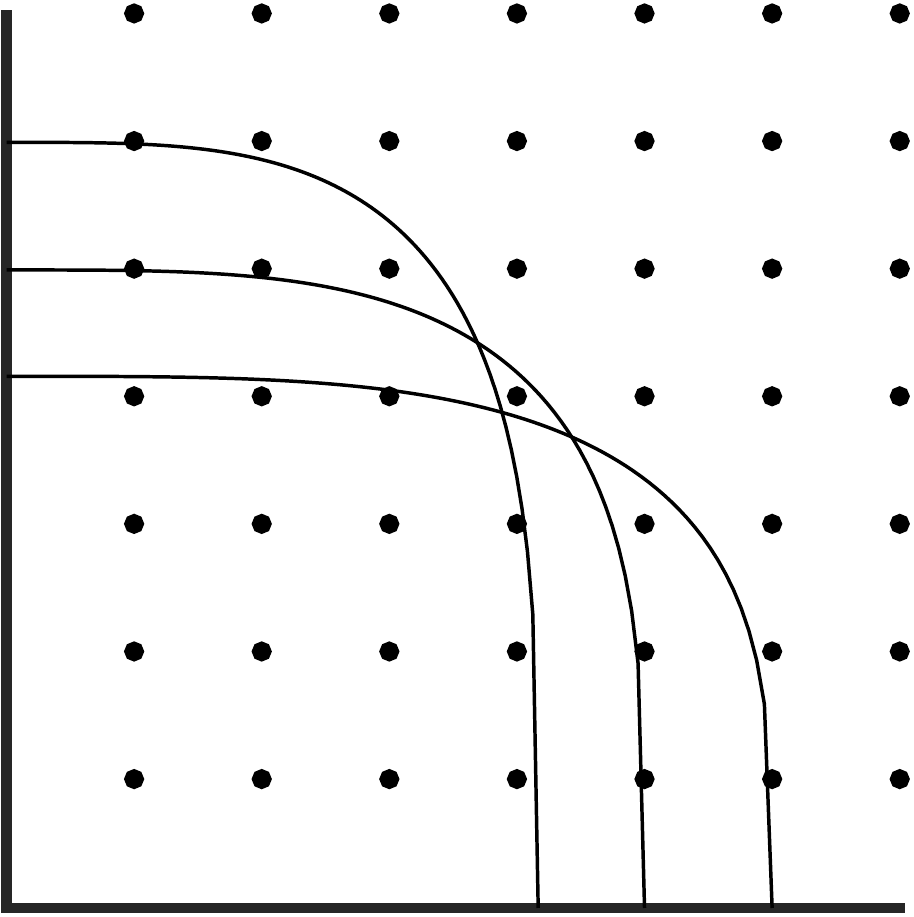}
\caption{Among all stretches of a concave curve, which captures the most
positive lattice points?}
\end{figure}

\subsection{Right triangles}
The results of \cite{AriturkLaugesen2017,LaugesenLiu2016} exclude the
$p=1$ case.  In contrast to other values of $p$, for the $p=1$ case (where
$p$-ellipses are right triangles) Ariturk \& Laugesen and Laugesen \& Liu
conjectured that the optimal triangle
\begin{quote}
``does not approach a 45-45-90 degree triangle as $r\rightarrow \infty$. Instead
one seems to get an infinite limit set of optimal triangles'' (from
\cite{LaugesenLiu2016}).
\end{quote}
The author and Steinerberger \cite{MarshallSteinerberger2017} recently proved
this conjecture and showed that the limit set is fractal of Minkowski dimension
at most $3/4$. Furthermore, all triangles which are optimal for infinitely many
infinitely large areas have rational slopes, are contained in $[1/3,3]$, and
have $1$ as a unique accumulation point. 

\begin{figure}[h!]
\centering
\begin{tabular}{cc}
\includegraphics[width=.2\textwidth]{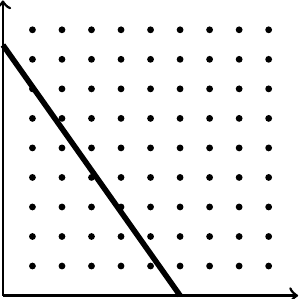}
&
\includegraphics[width=.55\textwidth]{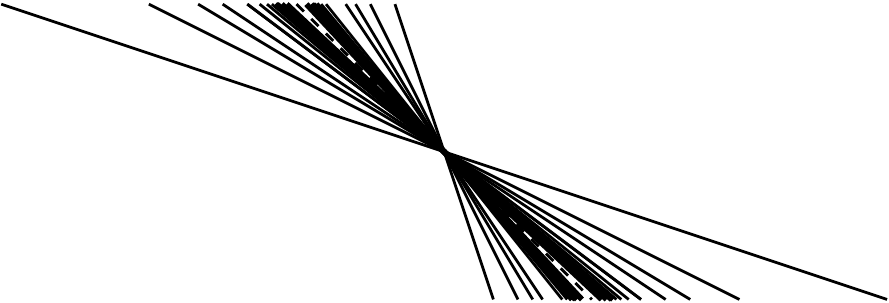}
\end{tabular}
\caption{An illustration of the fractal set of infinitely many times optimal
slopes.}
\end{figure}

\subsection{Relation to spectral asymptotics}
Suppose that $R \subset \mathbb{R}^d$ is a cuboid, and let $N_R(\lambda)$ denote
the number of eigenvalues of the Dirichlet Laplacian $-\Delta_R^\mathcal{D}$
that are less than $\lambda$ 
$$
N_R(\lambda) = \# \left\{ \mu \in \sigma \left( -\Delta_R^\mathcal{D} \right) :
\mu \le \lambda \right\}.
$$
Then $N_R(\lambda)$ has a two-term asymptotic formula in terms of the volume
$|R|$ and surface area $|\partial R|$ of $R$
$$
N_R(\lambda) = \frac{v_d}{(2\pi)^d} |R| \lambda^{d/2} - \frac{v_{d-1}}{4(2
\pi)^{d-1}} |\partial R|  \lambda^\frac{d-1}{2} + o_R\left(\lambda^\frac{d-1}{2}
\right), \quad \text{as} \quad \lambda \rightarrow \infty,
$$
where $v_d$ is the volume of the unit ball in $\mathbb{R}^d$. A similar
asymptotic formula holds for the eigenvalues of the Neumann Laplacian
$-\Delta_R^\mathcal{N}$, see \cite{Ivrii2016}.  We remark that in 1913, Weyl,
see page 199 of \cite{Weyl1913}, speculated about the existence of such a
two-term asymptotic formula for more general domains in $\mathbb{R}^3$, and the
problem of establishing the above two-term asymptotic formula for general
domains became known as Weyl's Conjecture, see \cite{Ivrii2016}.  In 1980, Ivrii
\cite{Ivrii1980} proved that this two-term asymptotic formula indeed holds for
more general domains under certain  conditions.  Suppose that $R$ is an $a_1
\times a_2 \times \cdots \times a_d$ cuboid. Then the eigenvalues of the
Dirichlet Laplacian $-\Delta^\mathcal{D}_R$ are of the form
$$
\sigma \left( - \Delta_R^\mathcal{D} \right) = \left\{
\pi^2 \left( \left(\frac{i_1}{a_1} \right)^2 + \left( \frac{i_2}{a_2} \right)^2
+ \cdots + \left( \frac{i_d}{a_d} \right)^2 \right): (i_1,i_2,\ldots,i_d) \in
\mathbb{Z}_{>0}^d \right\},
$$
and thus, the eigenvalues of $-\Delta_R^\mathcal{D}$ less than $\pi^2
r^2$ are in bijection with the positive lattice points in the ellipsoid
$(x_1/a_1)^2 + (x_2/a_2)^2 + \cdots + (x_d/a_d)^2 \le r^2$. If this cuboid $R$
has unit measure $|R|=1$, then its surface area $|\partial R|$ is given by 
$$
|\partial R| = \sum_{j=1}^d 2 a_j^{-1} = 2 \left( \tr A^{-1} \right), \quad
\text{where} \quad A = \diag(a_1,a_2,\ldots,a_d).
$$
Substituting $\lambda = \pi^2 r^2$ into the above two-term asymptotic
formula for $N_R(\lambda)$ for this cuboid gives
$$
N_R(\pi^2 r^2) = \frac{1}{2^d} v_d \, r^d - \frac{1}{2^d} v_{d-1} \left( \tr A^{-1}
\right)r^{d-1} + o_R\left(r^{d-1} \right), \quad \text{as} \quad r \rightarrow
\infty.
$$
We emphasize that the error term in this asymptotic formula depends implicitly
on the cuboid $R$, and therefore, this asymptotic formula by itself is
insufficient to determine which cuboid $R$ of unit measure maximizes
$N_R(\lambda)$ as $\lambda \rightarrow \infty$.  Addressing this issue is the
main challenge in \cite{AntunesFreitas2012, GittinsLarson2017,
vandenBergBucurGittins2016, vandenBergGittins2016}. If we were to ignore the
error term, then maximizing $N_R(\lambda)$ among cuboids of unit measure would
be equivalent to minimizing $\tr A^{-1}$. By the arithmetic mean geometric mean
inequality 
$$
\tr A^{-1} = d \cdot \frac{a_1^{-1} + a_2^{-1} + \cdots + a_d^{-1}}{d} \ge d
\cdot \left( a_1^{-1} \cdot a_2^{-1} \cdot \cdots \cdot a_d^{-1} \right)^{1/d} =
d,
$$
with equality if and only if $A = \Id$.  Since $2\tr A^{-1} = |\partial R|$,
this inequality can be interpreted as an isoperimetric inequality for
cuboids:  the surface area of a cuboid of unit measure is greater than or equal
to the surface area of the unit cube with equality if and only if the cuboid is
the unit cube. Ultimately, after the main challenge of dealing with the error
term has been appropriately handled, this isoperimetric inequality for cuboids
is the reason that the cube is asymptotically optimal in
\cite{AntunesFreitas2012, GittinsLarson2017, vandenBergBucurGittins2016,
vandenBergGittins2016}.  In this paper, we consider a positive lattice point
counting problem in more general domains; specifically, we study the number of
positive lattice points in a fixed convex domain $\Omega$ that has been scaled
by $r \ge 1$ and stretched by a linear transformation $A$ represented by a
positive diagonal matrix of determinant $1$.  In this generalized setting, the
term $\tr A^{-1}$ similarly arises, and we use Fourier analysis to develop
lattice point counting results which lead to uniform error estimates for optimal
stretching problems.

\subsection{Motivation}
In this paper, our motivation is twofold: first, the results of Laugesen \& Liu
and Ariturk \& Laugesen show that the asymptotic balancing observed in lattice
point problems for ellipses \cite{AntunesFreitas2012, GittinsLarson2017,
vandenBergBucurGittins2016, vandenBergGittins2016} extends from
ellipses to a more general class of convex and concave curves, at least in two
dimensions, and second, the analysis of the $p=1$ case in
\cite{MarshallSteinerberger2017} shows that the convergence breaks down when the
curves become flat. This phenomenon is common in harmonic analysis, where the
decay of the Fourier transform is dependent on non-vanishing curvature.  To
briefly review the relation of Fourier analysis to lattice point problems,
suppose $f$ is a $C^\infty$ function on $\mathbb{R}^d$ of compact support. Then
the Poisson summation formula states that
$$
\sum_{n \in \mathbb{Z}^d} f(n) = \sum_{n \in \mathbb{Z}^d} \widehat{f}(n),
$$
where $\widehat{f}$ denotes the Fourier transform of $f$
$$
\widehat{f}(\xi) = \int_{\mathbb{R}^d} f(x) e^{-2\pi i x \cdot \xi} dx.
$$
Suppose that $\chi(x)$ is the indicator function for a domain $\Omega \subset
\mathbb{R}^d$. Then the indicator function for the scaled domain $r \Omega$ is
$\chi_r(x) = \chi(x/r)$, whose Fourier transform, by a change of variables is
$$
\widehat{\chi}_r(\xi) = r^d \widehat{\chi}(r \xi).
$$
Moreover, since $\widehat{\chi}_r(0) = |\Omega| r^d$, if the Poisson summation
formula could be applied to $\chi_r$, it would express the number of lattice
points inside $r \Omega$ as $|\Omega| r^d$ plus the sum of $\widehat{\chi}_r$
over the nonzero lattice points.  Unfortunately, the Fourier transforms of
indicator functions do not decay rapidly enough for the Poisson summation
formula to be applied (these functions lack sufficient smoothness).  However,
this issue can be resolved by smoothing the indicator function by convolution
with a bump function, and useful estimates can be obtained. This approach can be
used to establish the classical 
$$
\#\{ (i,j) \in \mathbb{Z}^2 : i^2 + j^2 \le r^2 \} =  \pi r^2 +
\mathcal{O}(r^{2/3}),
$$
result for the Gauss circle problem accredited to Sierpi\'nski
\cite{Sierpinski1906}, van der Corput \cite{vanderCorput1923}, and Voronoi
\cite{Voronoi1903}, which was the first non-trivial step towards the
conjectured result:
$$
\# \{ (i,j) \in \mathbb{Z}^2 : i^2 + j^2 \le r^2 \} =  \pi r^2 +
\mathcal{O} \left(r^{1/2 + \varepsilon} \right),
$$
for all $\varepsilon > 0$. Currently the best known result is
$\mathcal{O}(r^{131/208})$ due to Huxley \cite{Huxley2003}.  The argument for
the $\mathcal{O}(r^{2/3})$ error term for the circle in $\mathbb{R}^2$ can be
generalized for convex domains $\Omega \subset \mathbb{R}^d$ whose boundary
$\partial \Omega$ has nowhere vanishing Gauss curvature \cite{Stein1993}. This
generalization results in the following bound on the number of enclosed lattice
points: 
$$
\# \{ n \in \mathbb{Z}^d \cap r\Omega\} = |\Omega| r^d + \mathcal{O} \left( r^{d - \frac{2d}{d+1}} \right).
$$
In this paper, we follow a similar approach to the proof of this result, but
additionally handle the effects of stretching the domain. More specifically,
given a strictly convex domain $\Omega$ which is symmetric with respect to each
coordinate hyperplane, we consider the positive lattice points contained in the
domain $A(r \Omega) = \{ A (rx) : x \in \Omega \}$, where $A$ is a positive
diagonal matrix of determinant $1$, and $r \ge 1$ is a scaling factor.
\begin{figure}[h!]
\centering
\includegraphics[width=.23\textwidth]{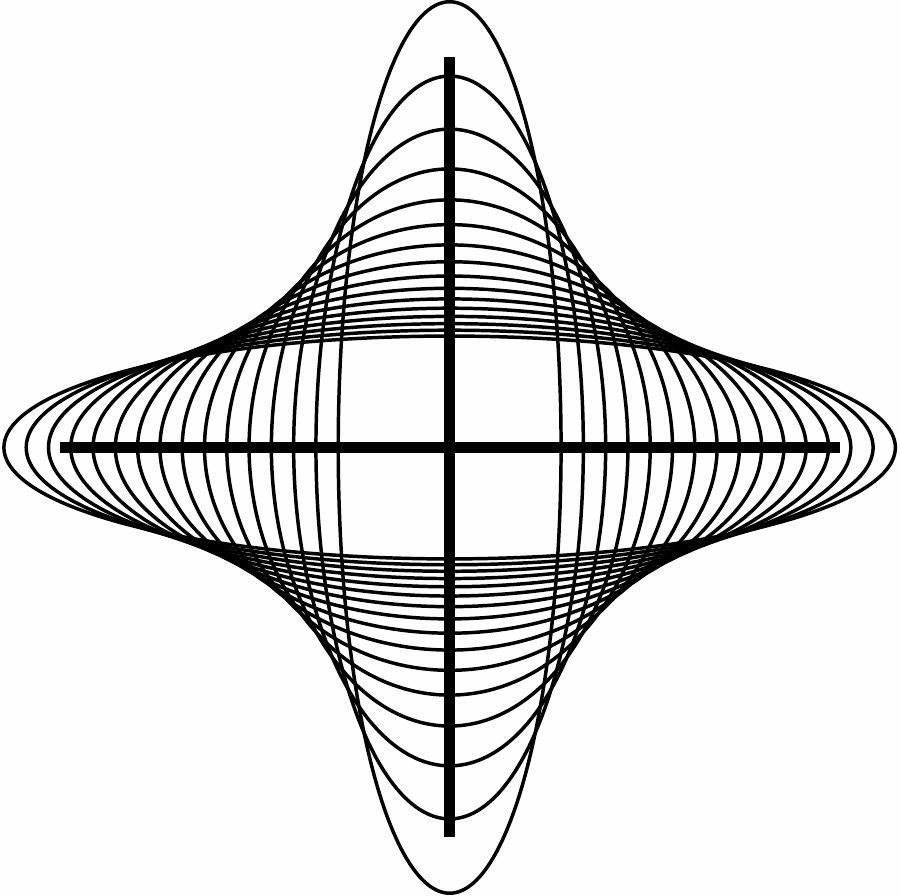}
\caption{The circle under the group action of positive diagonal matrices of
determinant $1$.}
\end{figure}

In the following we compute an asymptotic expansion for $\#\{ n \in \mathbb{Z}^d
\cap A(r\Omega) \}$ which includes both the effects of the diagonal
transformation $A$ and scaling factor $r$; we use this expansion to derive
uniform error estimates for an optimal stretching problem.  The resulting
theorem extends the results of \cite{AntunesFreitas2012, AriturkLaugesen2017,
GittinsLarson2017, vandenBergBucurGittins2016, vandenBergGittins2016}.  We note
that we do not completely recover the results of \cite{AriturkLaugesen2017,
LaugesenLiu2016}, since our approach can only represent curves which can be
realized as the boundary of a convex domain whose boundary has nowhere
vanishing Gauss curvature. However, there is some hope that the presented
framework could be used to fully generalize the results of Laugesen \& Liu
\cite{LaugesenLiu2016} and Ariturk \& Laugesen \cite{AriturkLaugesen2017} by
using more delicate bounds on the decay of the Fourier transform. Thus the
presented approach may be useful for proving further generalizations. 

\section{Main Result}

\subsection{Main Result} \label{mainresult}
Suppose that $\Omega \subset \mathbb{R}^d$ is a bounded convex domain whose
boundary $\partial \Omega$ is $C^{d+2}$ and has nowhere vanishing Gauss
curvature. Furthermore, suppose that $\Omega$ is symmetric with respect to each
coordinate hyperplane and balanced in the following sense.

\begin{definition} \label{balanced}
We say that a bounded convex domain $\Omega \subset \mathbb{R}^d$ is balanced
if 
$$
\left| \left\{ (x_1,x_2,\ldots,x_d) \in \Omega : x_j = 0 \right\}\right| = 1
\quad \text{for } j =1,2,\ldots,d, 
$$
where $|\cdot|$ denotes the $(d-1)$-dimensional Lebesgue measure. 
\end{definition}

Note that the balanced assumption does not restrict the domains for which the
below Theorem applies. Rather, the assumption that $\Omega$ is balanced is
equivalent to choosing the unique balanced representative $B \Omega$ for
$\Omega$, where $B$ is a positive diagonal matrix. 
Suppose that $A = \diag(a_1,a_2,\ldots,a_d)$ is a positive definite diagonal
matrix of determinant $1$.  To reiterate, we define
$$
A(r \Omega) = \{ A (rx) : x \in \Omega \},
$$
for scaling factor $r \ge 1$.  That is to say, $A(r \Omega)$ is the domain
$\Omega$ scaled by $r$ and transformed by $A$. 
\begin{figure}[h!]
\centering
\begin{tabular}{cc}
\includegraphics[width=.2\textwidth]{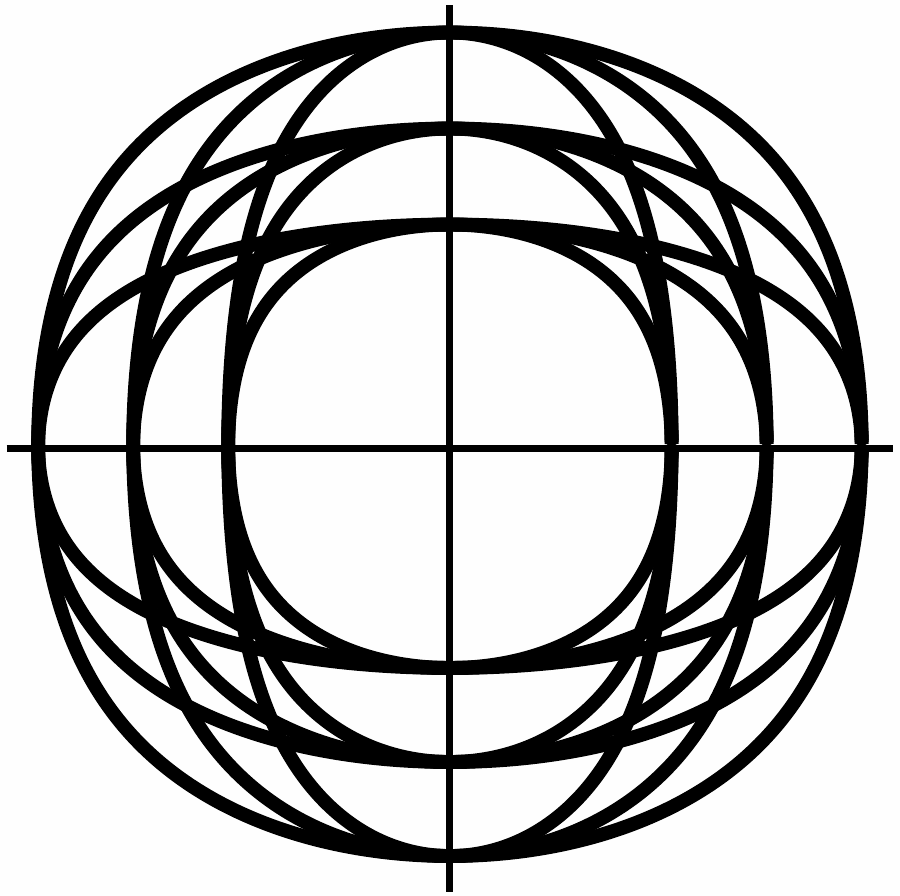} &
\includegraphics[width=.2\textwidth]{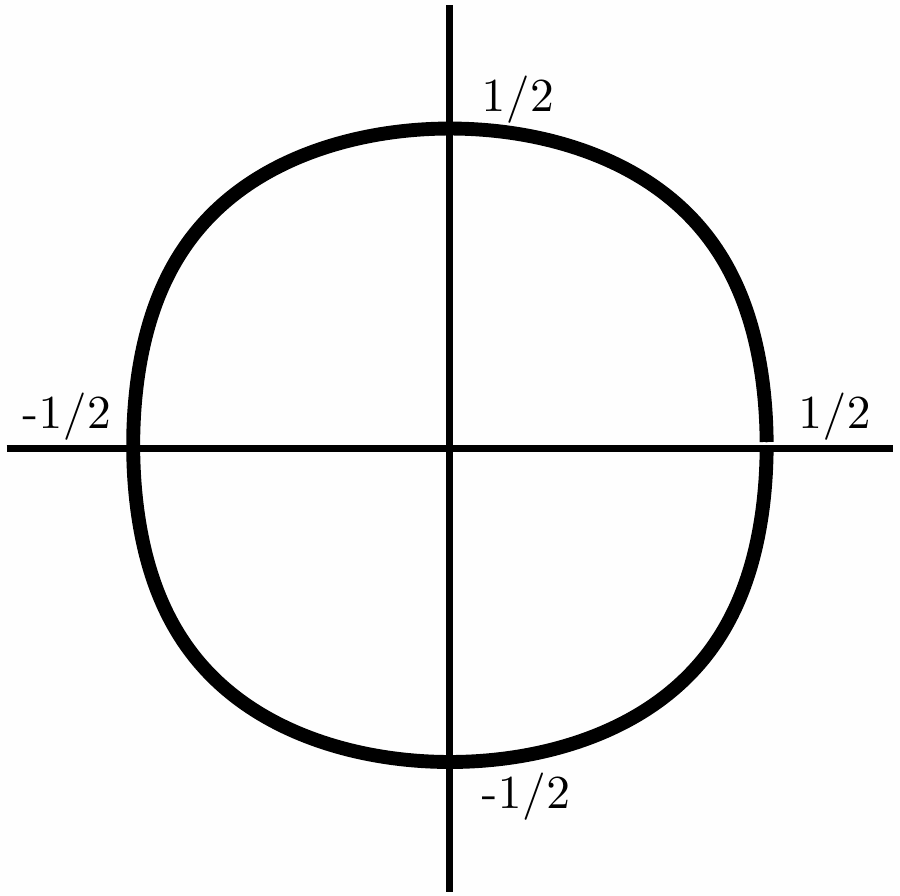} 
\end{tabular}
\caption{Each domain $\Omega \subset \mathbb{R}^d$ which contains the origin has
a unique balanced representative; the domains illustrated on the left share
the balanced representative illustrated on the right.}
\end{figure}

\begin{theorem} \label{thm1} Suppose $\Omega \subset \mathbb{R}^d$
and $A : \mathbb{R}^d \rightarrow \mathbb{R}^d$ are as described above. Then for
$d \ge 2$ and $r \ge 1$,
$$
\#\left\{ n \in \mathbb{Z}^d_{>0} \cap A(r\Omega) \right\}= \frac{1}{2^d}
|\Omega| r^d - \frac{1}{2^{d}} \left(\tr A^{-1} \right) r^{d-1} + \mathcal{O}
\left( \|A^{-1}\|^\frac{2d}{d+1}_\infty r^{d - \frac{2d}{d+1}} \right),
$$
where the implicit constant is independent of $A$ and $r$. Moreover, if
$$
A(r) = \argmax_A \, \# \left\{n \in \mathbb{Z}^d_{>0} \cap A(r\Omega) \right\},
$$
then
$$
\| A(r) - \Id \|_\infty =  \mathcal{O}\left( r^{-\frac{d-1}{2(d+1)}} \right)
\quad \text{as } r \rightarrow \infty,
$$
where the implicit constant is independent of $A$ and $r$.
\end{theorem}

In the following remark, we describe specifically how the implicit constants in
Theorem \ref{thm1} depend on $\Omega \subset \mathbb{R}^d$, and provide
conditions under which the expansion and convergence rate results in Theorem
\ref{thm1} hold uniformly over a family of domains.

\begin{remark}
The implicit constants in Theorem \ref{thm1} which depend on the domain $\Omega
\subset \mathbb{R}^d$ can be chosen in terms of the following three quantities:
a lower bound on the Gauss curvature of $\Omega$, an upper bound on the diameter
of $\Omega$, and a lower bound on the inradius of $\Omega$. Therefore, the
result of Theorem \ref{thm1} holds uniformly over any family of admissible
domains, which have uniform bounds for these three quantities.  Indeed,
constants depending on $\Omega \subset \mathbb{R}^d$ enter the proof of Theorem
\ref{thm1} from three sources. First, we use a constant $C > 0$ such that
$\Omega \subset [-C,C]^d$; since $\Omega$ contains the origin it suffices to
choose $C$ equal to the diameter of $\Omega$. Second, we use a constant $c > 0$
from Lemma \ref{lem:geo2}; by the proof of Lemma \ref{lem:geo2}, this constant
can be chosen as one divided by the inradius of $\Omega$. Third, we implicitly
use a constant when using the bound in Lemma \ref{lem:decay} for the decay of
the Fourier transform of the indicator function for $\Omega$; this implicit
constant can be chosen in terms of a lower bound on the Gauss curvature of the
domain, cf.  \cite{Stein1993}.  
\end{remark}

When $\Omega \subset \mathbb{R}^d$ is a $d$-dimensional ellipsoid, the Theorem
implies the Dirichlet Laplacian results of Gittins \& Larson
\cite{GittinsLarson2017}.  Applying the Theorem for dimension $d=2$ recovers the
original result of Antunes \& Freitas \cite{AntunesFreitas2012} and agrees with
the error estimate in \cite{LaugesenLiu2016}. Specifically, in dimension $d=2$,
the set of all positive diagonal matrices of determinant $1$ is the
$1$-parameter family $A = \diag(1/a,a)$ for $a >0$ and the result of the Theorem
can be stated as:

\begin{corollary}
In the case $\Omega \subset \mathbb{R}^2$, the Theorem gives
$$
\#\left\{ n \in \mathbb{Z}^2_{>0} \cap A(r \Omega) \right\}  = \frac{1}{4} |\Omega| r^2 -
\frac{1}{4} \left( a + \frac{1}{a} \right) r + \mathcal{O} \left( a^{4/3}
r^{2/3} \right) ,
$$
and 
$$
| a(r) - 1 | = \mathcal{O}\left( r^{-1/6} \right).
$$
\end{corollary}

As a second Corollary we can establish dual results for the nonnegative lattice
point problem, which corresponds to the Neumann Laplacian results of
Gittins \& Larson \cite{GittinsLarson2017}.

\begin{corollary}
Suppose $\Omega \subset \mathbb{R}^d$ and $A : \mathbb{R}^d \rightarrow
\mathbb{R}^d$ satisfy the hypotheses of the Theorem, and additionally 
suppose $1 \le \|A^{-1}\|_\infty \le C r$ for a fixed constant $C > 0$.
Then
$$
\# \left\{ n \in \mathbb{Z}^d_{\ge 0} \cap A(r\Omega) \right\}= \frac{1}{2^d}
|\Omega| r^d + \frac{1}{2^{d}} \left(\tr A^{-1} \right) r^{d-1} + \mathcal{O}
\left( \|A^{-1}\|^\frac{2d}{d+1}_\infty r^{d - \frac{2d}{d+1}} \right),
$$
where the implicit constant is independent of $A$ and $r$. Moreover, if
$$
A(r) = \argmin_A \, \#\{n \in \mathbb{Z}^d_{\ge 0} \cap A(r\Omega) \},
$$
then
$$
\| A(r) - \Id \|_\infty =  \mathcal{O}\left( r^{-\frac{d-1}{2(d+1)}} \right),
$$
where the implicit constant is independent of $A$ and $r$.
\end{corollary}

The  assumption $1 \le \|A^{-1}\|_\infty \le C r$ is needed for the expansion to
hold, but not necessary for the convergence results as it serves to avoid the
case where $A(r \Omega)$ contains more than order $r^d$ lattice points on the
coordinate hyperplanes which is clearly non-optimal for the $\argmin$. The proof
of this Corollary follows from the arguments in Step \ref{step3} of the proof of
Theorem \ref{thm1}.

\subsection{Generalization}
In the following we describe a generalization of Theorem \ref{thm1}. In
particular, we remark how Theorem \ref{thm1} can be generalized to domains
$\Omega$ which are not necessarily symmetric with respect to each coordinate
hyperplane by considering the lattice points $\{ n \in \mathbb{Z}^d : n_j \not =
0, \forall j = 1,\ldots,d\}$ rather than the positive lattice points
$\mathbb{Z}^d_{>0}$.
 
\begin{remark} \label{gen}
Suppose that $\Omega \subset \mathbb{R}^d$ is a bounded convex domain whose
boundary $\partial \Omega$ is $C^{d+2}$ and has nowhere vanishing Gauss
curvature. Moreover, suppose $\Omega$ is balanced in the sense of Definition
\ref{balanced}, and contains the origin. Define $\mathbb{Z}^d_{\not = 0} :=
\{ n \in \mathbb{Z}^d : n_j \not = 0, \, \forall j = 1,\ldots,d\}$. Let $A =
\diag(a_1,a_2,\ldots,a_d)$ be a positive diagonal matrix of determinant $1$.
Then
$$
\# \left\{ n \in \mathbb{Z}^d_{\not=0} \cap A \left(r \Omega \right)
\right\} = |\Omega| r^d - \left( \tr A^{-1} \right) r^{d-1} + \mathcal{O} \left(
\|A^{-1}\|^\frac{2d}{d+1}_\infty r^{d - \frac{2d}{d+1}} \right), 
$$
where the implicit constant is independent of $A$ and $r$. Moreover, if
$$
A(r) = \argmax_A \, \# \left\{n \in \mathbb{Z}^d_{\not=0} \cap A
\left(r\Omega \right) \right\},
$$
then
$$
\| A(r) - \Id \|_\infty =  \mathcal{O}\left( r^{-\frac{d-1}{2(d+1)}} \right),
$$
where the implicit constant is independent of $A$ and $r$.  The proof of this
statement is immediate from Step \ref{step3} of the proof of Theorem \ref{thm1}.
\end{remark}

\section{Proof of main result}

\subsection{Proof strategy}
Before discussing the technical details, we describe the proof strategy.
The proof is divided into five steps. First, we establish a new lattice point
counting result which holds uniformly over a family of positive diagonal
matrices $A$ with determinant $1$. Specifically, we show  that
$$
\# \{ n \in \mathbb{Z}^d \cap A(r \Omega) \} = |\Omega| r^d + \mathcal{O} \left(
\|A^{-1}\|_\infty^\frac{2d}{d+1} r^{d - \frac{2d}{d+1}} \right),
$$
when $1 \le \|A^{-1}\|_\infty \le C r$ where $C > 0$ is a fixed constant.  The
key observation in the proof of this expansion is that the classical Fourier
analysis techniques used to study the Gauss circle problem can be applied if the
indicator function of the domain $A(r \Omega)$ is mollified by a bump function
which has been appropriately stretched and scaled.  Second, as a consequence of
this result we show that the number of lattice points on the coordinate
hyperplanes is, as expected, equal to $\left( \tr A^{-1} \right) r^{d-1}$ plus
an appropriate error term; specifically, 
$$
\# \left\{ n \in \mathbb{Z}^d \cap A  \left(r \bigcup_{j=1}^d \Omega_j \right)
\right\} = (\tr A^{-1}) r^{d-1} + \mathcal{O}
\left(\|A^{-1}\|_\infty^{\frac{2d}{d+1}} r^{d-\frac{2d}{d+1}} \right),
$$
where $\Omega_j$ denotes the intersection of $\Omega$ with the hyperplane
orthogonal to the $j$-th coordinate vector.  Third, we combine these two results
to produce a positive lattice point counting result for $A(r\Omega)$:
$$
\# \left\{ n \in \mathbb{Z}_{>0}^d \cap A(r \Omega) \right\} = 
\frac{1}{2^d} |\Omega| r^d - \frac{1}{2^d} \left( \tr A^{-1} \right) r^{d-1} + 
\mathcal{O} \left(\|A^{-1}\|_\infty^{\frac{2d}{d+1}} r^{d - \frac{2d}{d+1}} \right).
$$
Fourth, we show that this positive lattice point counting result implies that
$$
\|A(r) - \Id \| \rightarrow 0, \quad \text{as } r \rightarrow \infty,
\quad
\text{where} 
\quad
A(r) = \argmax_A \, \# \left\{n \in \mathbb{Z}^d_{>0} \cap A(r\Omega) \right\},
$$
where the $\argmax$ is taken over positive diagonal matrices $A$ of determinant
$1$. We note that clearly $\|A(r)^{-1}\|_\infty = \mathcal{O}(r)$ since
otherwise $A(r \Omega)$ will not contain any positive lattice points; we 
proceed by considering two cases. First, we consider the case
where
$$
\|A^{-1}\|_\infty \quad \text{is on the order of} \quad r,
$$
where our positive lattice point counting result for $A(r \Omega)$
provides no information because the error term is order $r^d$.  However, in such
an extreme situation we are able to independently show non-optimality. Second,
we assume that 
$$
\|A^{-1}\|_\infty = o(r),
$$
and use the positive lattice point counting result to show convergence.
Specifically, we assume
$$
\|A^{-1}\|_\infty = \psi(r) r, \quad \text{where} \quad \psi(r) \rightarrow
0, \quad \text{as} \quad r \rightarrow \infty.
$$ 
Substituting $\|A^{-1}\|_\infty = \psi(r) r$ into the positive lattice point
counting result for $A(r\Omega)$ gives
$$
\# \left\{ n \in \mathbb{Z}_{>0}^d \cap A(r\Omega) \right\} \le \frac{1}{2^d}
|\Omega| r^d - \left(\frac{1}{2^{d}} - C \psi(r)^\frac{d-1}{d+1} \right)
\psi(r) r^d,
$$
where $C$ is a fixed constant. Since $\psi \rightarrow 0$ the term in the
parentheses we eventually be positive, and hence, comparison to the case where
$A = \Id$ implies convergence. Fifth, and finally, given the fact that $A(r)$
converges to $\Id$ we establish the convergence rate
$$
\|A(r) - \Id\|_\infty = \mathcal{O}\left( r^{-\frac{d-1}{2(d+1)}} \right),
$$
using a standard arithmetic mean geometric mean argument.

\subsection{Useful lemmata} The proof of the main result relies on three
lemmata: two geometric in nature, and one related to the decay of the Fourier
transform of indicator functions of convex domains with nowhere vanishing Gauss
curvature. Lemmata \ref{lem:geo2} and \ref{lem:new} are proved in this section,
and are motivated by similar technical results in \cite{SteinShakarchi2011},
while Lemma \ref{lem:decay} appears in standard references, cf.,
\cite{Stein1993, SteinShakarchi2011, Travaglini2004}.

\begin{lemma} \label{lem:geo2}
Suppose that $\Omega \subset \mathbb{R}^d$ is a bounded convex open domain which
contains the origin. Then there exists a constant $c > 0$ such that for all $0 <
r < \infty$, all $0 < \delta < \infty$, and all $0 \le |y| \le \delta$ 
$$
x \in r \Omega \implies x + y \in (r + c \delta) \Omega.
$$
\end{lemma}
\begin{proof}
Since $\Omega$ is open, bounded, and contains the origin $\vec{0} \in
\mathbb{R}^d$, there exists a constant $\varepsilon > 0$ such that
$$ 
\{ x \in \mathbb{R}^d : |x - \vec{0}| \le \varepsilon \} \subset \Omega.
$$
Set $c = 1/ \varepsilon$, and suppose that $0 < r < \infty$, $x \in r \Omega$,
and $0 < |y| \le \delta$ are given. Then
$$
y = \frac{|y|}{\varepsilon} \cdot \varepsilon \frac{y}{|y|} \in \{ x \in
\mathbb{R}^d : |x - \vec{0}| \le c \delta \varepsilon \} \subseteq c \delta
\Omega,
$$
and therefore,
$$
x + y \in r \Omega + c \delta \Omega \subseteq (r + c \delta) \Omega,
$$
where the final inclusion follows from the convexity of $\Omega$.
\end{proof}
\begin{figure}[h!]
\centering
\includegraphics[width=.19\textwidth]{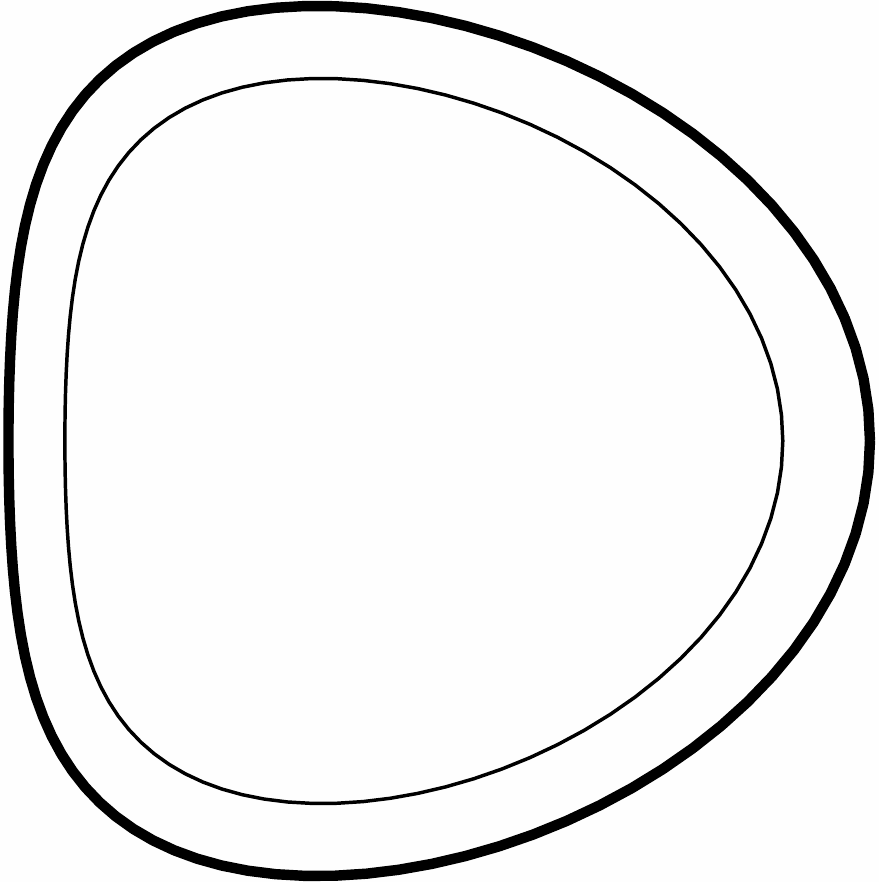}
\caption{Geometrically Lemma \ref{lem:geo2} implies that $\{x \in \mathbb{R}^d :
d(x,r\Omega) \le \delta\} \subseteq (r+c\delta) \Omega$.}
\end{figure}

The application of this Lemma to lattice point counting problems occurs when
developing bounds for indicator functions in terms of mollified indicator
functions.  Specifically, Lemma \ref{lem:geo2} is typically applied as follows.  Suppose
$\chi$ is the indicator function for the set $\Omega \subset \mathbb{R}^d$,
and define
$$
\chi_r(x) = \chi(x/r),
$$
which is the indicator function for $r \Omega$.  Let
$\varphi$ denote a $C^\infty$ bump function supported on the unit ball which
integrates to $1$. Set
$$
\varphi_\delta(x) = \delta^{-d} \varphi(x/\delta).
$$
Suppose $c > 0$ is chosen in accordance to Lemma \ref{lem:geo2}. Then we claim
that for all $x \in
\mathbb{R}^d$
$$
\chi_r(x) \le [\chi_{r + c \delta} * \varphi_\delta] (x),
$$
where $*$ denotes convolution. Indeed, by the choice of $c > 0$, for all
$|y| \le \delta$
$$
\chi_r(x) = 1 \implies \chi_{r+c\delta}(x-y) = 1.
$$
Therefore, if $\chi_r(x) = 1$, then
$$
[\chi_{r + c \delta} * \varphi_\delta] (x) = \int_{|y| \le \delta}
\chi_{r+c\delta}(x-y) \varphi_\delta(y) dy =  \int_{|y| \le \delta}
\varphi_\delta(y) dy  = 1.
$$
In the proof of the Theorem in the following section a similar result is
required.  However, in this case the indicator functions and bump functions
under consideration are transformed by a positive diagonal linear transformation
$A$ of determinant $1$. More precisely, we consider 
$$ 
\chi_{A,r}(x) = \chi(A^{-1} x/r) \quad \text{and} \quad \varphi_{A,\delta} =
\delta^{-d} \varphi(A^{-1}x/\delta),
$$ 
where $A$ is a positive diagonal matrix of determinant $1$. Let
$$
\chi_{A,r,\delta} = \chi_{A,r} * \varphi_{A,\delta}.
$$
In the following Lemma, we repeat the above analysis to develop upper and lower
bounds for $\chi_{A,r}$  in terms of the smoothed version $\chi_{A,r,\delta}$.

\begin{lemma} \label{lem:new}
Suppose $\chi_{A,r}$ and $\chi_{A,r,\delta}$ are as defined above. Then there
exists a constant $c > 0$ such that for all $r > 0 $, all $\delta \le r/(1+c)$,
and all positive diagonal matrices $A$ of determinant $1$
$$
\chi_{A,r - c \delta,\delta}(x) \le \chi_{A,r}(x) \le \chi_{A,r + c
\delta,\delta}(x),
$$
for all $x \in \mathbb{R}^d$. 
\end{lemma}
\begin{proof}
By Lemma \ref{lem:geo2} we may choose a constant $c > 0$ such that for all $r >
0$ and all $\delta > 0$
$$
x \in r \Omega \wedge |y| \le \delta \implies x + y \in (r + c \delta) \Omega.
$$
To establish the upper bound it suffices to show that $\chi_{A,r}(x) = 1
\implies \chi_{A,r + c \delta,\delta} = 1$. By definition
$$
\chi_{A,r + c \delta,\delta}(x) = \int_{|A^{-1}y| \le \delta} \chi_{r+c\delta}(A^{-1}(x-y))
\varphi_\delta(A^{-1}y) dy.
$$
Since $A$ is a diagonal matrix of determinant $1$, by a change of variables of
integration we conclude
$$
\int_{|A^{-1}y| \le \delta} \chi_{r+c\delta}(A^{-1}(x-y))
\varphi_\delta(A^{-1}y) dy  = \int_{|y| \le \delta}
\chi_{r+c\delta}(A^{-1}x-y) \varphi_\delta(y) dy.
$$
If $\chi_{A,r}(x) = 1$, then by the choice of $c>0$ we have
$\chi_{r+c\delta}(A^{-1}x-y) =1$, when $|y| \le \delta$. Thus
$$
\int_{|y| \le \delta} \chi_{r+c\delta}(A^{-1}x-y) \varphi_\delta(y) dy = 
\int_{|y| \le \delta} \varphi_\delta(y) dy = 1,
$$
which establishes the upper bound. Next, to establish the lower bound we show
that $\chi_{A,r}(x) = 0 \implies \chi_{A,r- c \delta,\delta}(x) = 0$. Let
$\tilde{x} = x- y$ and $\tilde{r} = r - c \delta$.  By Lemma \ref{lem:geo2}
$$
\tilde{x} \in \tilde{r} \Omega \wedge |y| \le \delta \implies \tilde{x} + y \in
(\tilde{r} + c \delta) \Omega.
$$
Writing the contrapositive of this statement in terms of $x$ and $r$ gives
$$
x \not \in r \Omega  \wedge |y| \le \delta  \implies x - y \not \in (r - c
\delta) \Omega.
$$
Using the same change of variables of integration as above, we have
$$
\chi_{A,r-c\delta,\delta}(x) = \int_{|y|\le \delta} \chi_{r -c \delta}(A^{-1} x
- y) \varphi_\delta(y) dy.
$$
If $\chi_{A,r}(x) = 0$, then the derived implication implies that $\chi_{r-c\delta}(A^{-1} x - y) = 0$ for all $|y| \le \delta$; thus
$$
\int_{|y|\le \delta} \chi_{r -c \delta}(A^{-1} x -
y) \varphi_\delta(y) dy = 0,
$$
which completes the proof.
\end{proof}

\begin{lemma} \label{lem:decay}
Suppose $\Omega \subset \mathbb{R}^d$ is a bounded convex domain with $C^{d+2}$
boundary $\partial \Omega$ with nowhere vanishing Gauss curvature. If $\chi$ is
the indicator function of $\Omega$, then
$$
|\widehat{\chi}(\xi)| = \mathcal{O} \left( |\xi|^{-\frac{d+1}{2}} \right),
\quad \text{as } |\xi| \rightarrow \infty.
$$
\end{lemma}

That is to say, the Fourier transform of an indicator function decays one order
better than the Fourier transform of the corresponding surface carried measure
$d\mu$, which decays like
$$
|\widehat{d\mu}| = \mathcal{O} \left( |\xi|^{-\frac{d-1}{2}} \right).
$$
The proof of this lemma involves an integration by parts of an expression
for the Fourier transform of the corresponding surface
carried measure, which is where the extra order of convergence arises.

\subsection{Proof of Theorem \ref{thm1}}
For clarity, we have divided the proof of the Theorem into five steps.
First, we fix notation. We say
$$
f(x) \lesssim_h g(x)
\quad \text{if and only if} \quad
f(x) \le C_h g(x),
$$
for a fixed constant $C_h >0$ only depending on $h$. Throughout the proof, we
assume that $r \ge 1$, $d \ge 2$, and that $\Omega \subset
\mathbb{R}^d$ is a bounded convex domain with a $C^{d+2}$ boundary $\partial
\Omega$ with nowhere vanishing Gauss curvature. In particular, we assume that
$$
\Omega \subset [-C,C]^d \quad \text{where} \quad C > 0 \quad \text{is a fixed
constant.}
$$
Let $A = \diag(a_1,a_2,\ldots,a_d)$ denote a positive diagonal matrix of
determinant $1$. Without loss of generality we may assume that $a_1 \le a_2 \le
\ldots \le a_d$, and set $a = 1/a_1 = \|A^{-1}\|_\infty$.

%\subsubsection*{Case 3} $C r < \delta$. Suppose that $C r < \delta$. We assert
%that in this case
%$$
%\# \{ n \in \mathbb{Z}^d \cap A (r \Omega) \} = 0.
%$$
%Recall that we have assumed $\Omega \subseteq [-C,C]^d$, and recall we assumed
%$A = \diag(a_1,a_2,\ldots,a_n)$, where $a_1 \le a_2 \le \cdots \le a_n$ and 
%$1/a_1 = a = \|A^{-1}\|_\infty$. Write $(x_1,x^\prime) \in \mathbb{R} \times
%\mathbb{R}^{d-1}$, and let $A_1$ denote the $(d-1) \times (d-1)$ matrix formed
%by removing the $1$-st row and $1$-st column from $A$. With this notation:
%$$
%(x_1,x^\prime) \in
%A(r \Omega) \implies \left(\frac{a x_1}{r},\frac{A_1^{-1} x^\prime}{r}
%\right) \in \Omega \subseteq [-C,C]^d.
%$$ 
%Therefore,
%$$
%|x_1| \le \frac{C r}{a} < 1,
%$$
%which implies $A(r\Omega)$ contains no positive lattice points.
%Moreover, 
%$$ 
%Cr < \delta \implies r C^\frac{d+1}{2d} < a.
%$$
%Hence the error term in the statement of the Lemma dominates terms of order
%$r^d$
%$$
%r^d \lesssim_\Omega a^\frac{2d}{d+1} r^{d - \frac{2d}{d+1}},
%$$
%from which the result follows.

\begin{step} \label{step1}
Suppose $1 \le a \le C r$. Then
$$
\# \{ n \in \mathbb{Z}^d \cap A(r \Omega) \} = |\Omega| r^d +
\mathcal{O} \left( a^\frac{2d}{d+1} r^{d - \frac{2d}{d+1}} \right),
$$
where the implicit constant is independent of $A$ and $r$.
\end{step} 
\begin{proof}
Let $\chi$ denote the indicator function for the given domain $\Omega$, and set
$$
\chi_{A,r,\delta}(x) = [\chi_{A,r} * \varphi_{A,\delta}](x),
$$
where $\chi_{A,r} = \chi(A^{-1} x /r)$ and $\varphi_{A,\delta} = \delta^{-d}
\varphi(A^{-1} x /\delta)$ for some $C^\infty$ bump function $\varphi$ supported
on the unit ball in $\mathbb{R}^d$ and such that
$$
\int_{\mathbb{R}^d} \varphi(x) dx = 1.
$$
We denote our ``smoothed'' approximation of the number of lattice points
enclosed by $A(r \Omega)$ by
$$
N_{A,r,\delta} = \sum_{n \in \mathbb{Z}^d} \chi_{A,r,\delta} (n).
$$
Since $\chi_{A,r,\delta}$ is $C^\infty$ and of compact support, by the Poisson
summation formula
$$
N_{A,r,\delta} = \sum_{n \in \mathbb{Z}^d} \widehat{\chi}_{A,r,\delta}(n) = \sum_{n \in
\mathbb{Z}^d} r^d \widehat{\chi}(A n r) \widehat{\varphi}(A n \delta).
$$
Since $\widehat{\chi}(0) = |\Omega|$ and $\widehat{\varphi}(0) = 1$, we can break the
sum up into three parts:
$$
N_{A,r,\delta} = |\Omega| r^d + \sum_{0 < |n| < a/\delta}  r^d \widehat{\chi}(A n r) \widehat{\varphi}(A n \delta)
+ \sum_{|n| \ge a/\delta}  r^d \widehat{\chi}(A n r) \widehat{\varphi}(A n \delta).
$$
Since $\widehat{\chi}$ is the Fourier transform of an indicator function of a convex
domain with nowhere vanishing Gauss curvature, it follows from Lemma
\ref{lem:decay} that
$$
|\widehat{\chi}(\xi)| \lesssim_\Omega |\xi|^{- \frac{d+1}{2}} .
$$
Furthermore, since $\widehat{\varphi}$ is the Fourier transform of a $C^\infty$
function with compact support if follows that
$$
|\widehat{\varphi}(\xi)| \lesssim_\varphi |\xi|^{-N},
$$
for all $N \ge 0$. To evaluate the first sum we use the fact that
$|\widehat{\chi}|  \lesssim_\Omega |\xi|^{- \frac{d+1}{2}} $, and
$|\widehat{\varphi}| \lesssim_\varphi 1$
$$
\sum_{0 < |n| \le a/\delta}  r^d \widehat{\chi}(A n r) \widehat{\varphi}(A n
\delta) \lesssim_{\Omega,\varphi}  \sum_{0 < |n|
\le a/\delta} r^d r^{-\frac{d+1}{2}} a^{\frac{d+1}{2}} |n|^{-\frac{d+1}{2}} .
$$
Therefore,
$$
\sum_{0 < |n| \le a/\delta}  r^d \widehat{\chi}(A n r) \widehat{\varphi}(A n
\delta) \lesssim_{\Omega,\varphi} r^{\frac{d-1}{2}} a^\frac{d+1}{2} \sum_{0 < |n|
\le a/\delta} |n|^{-\frac{d+1}{2}} .
$$
The sum on the right hand side can be compared to the integral
$$
\int_{|x| \le a/\delta} |x|^{-\frac{d+1}{2}} dx = C_d \int_{0}^{a/\delta}
t^{-\frac{d+1}{2}} t^{d-1} dt = C_d a^\frac{d-1}{2} \delta^{-\frac{d-1}{2}}.
$$
Thus we conclude that
$$
\sum_{0 < |n| \le a/\delta}  r^d \widehat{\chi}(A n r) \widehat{\varphi}(A n
\delta) \lesssim_{\Omega,\varphi} r^\frac{d-1}{2} a^d
\delta^{-\frac{d-1}{2}}.
$$
For the second sum, we use the fact that
$$
|\widehat{\varphi}(\xi)| \lesssim_{\varphi} |\xi|^{-\frac{d}{2}} ,
$$
and use the same decay of $\widehat{\chi}$ as before. This yields
$$
\sum_{|n| \ge a/\delta}  r^d \widehat{\chi}(A n r) \widehat{\varphi}(A n
\delta) \lesssim_{\Omega,\varphi} \sum_{|n| \ge a/\delta} r^d
r^{-\frac{d+1}{2}} a^{\frac{d+1}{2}} |n|^{-\frac{d+1}{2}} a^{\frac{d}{2}}
|n|^{-\frac{d}{2}} \delta^{-\frac{d}{2}} .
$$
Hence
$$
\sum_{|n| \ge a/\delta}  r^d \widehat{\chi}(A n r) \widehat{\varphi}(A n
\delta) \lesssim_{\Omega,\varphi} r^{\frac{d-1}{2}} \delta^{-\frac{d}{2}}
a^{d + \frac{1}{2}} \sum_{|n| \ge a/\delta} |n|^{-d - \frac{1}{2}} .
$$
We can bound the sum on the right hand side by comparison to the integral
$$
\int_{|x|\ge a/\delta} |x|^{-d-1/2} dx = C_d \int_{a/\delta}^{\infty} t^{-d
-\frac{1}{2}}
t^{d-1} dt = C_d a^{-\frac{1}{2}} \delta^{\frac{1}{2}} .
$$
So 
$$
\sum_{|n| \ge a/\delta}   r^d \widehat{\chi}(A n r) \widehat{\varphi}(A n
\delta) \lesssim_{\Omega,\varphi}  r^{\frac{d-1}{2}} a^d
\delta^{-\frac{d-1}{2}} .
$$
Moreover, since the bound for both sums is the same
$$
\left| N_{A,r,\delta} - |\Omega| r^d \right|  \lesssim_{\Omega,\varphi}
r^\frac{d-1}{2} a^d \delta^{-\frac{d-1}{2}}.
$$
Set
$$
\delta = a^{\frac{2d}{d+1}} r^{-\frac{d-1}{d+1}}.
$$
Substituting this value of $\delta$ into the last inequality yields
$$
\left| N_{A,r,\delta} - |\Omega| r^d \right|  \lesssim_{\Omega,\varphi}
 a^\frac{2d}{d+1} r^{d - \frac{2d}{d+1}}.
$$
By Lemma \ref{lem:new} there exists a constant $c > 0$ such that for all $r > 0
$, all $\delta \le r/(1+c)$, and all positive diagonal matrices $A$ of
determinant $1$
$$
\chi_{A,r - c \delta,\delta}(x) \le \chi_{A,r}(x) \le \chi_{A,r + c \delta}(x),
$$
for all $x \in \mathbb{R}^d$. However, we do not in general know that $\delta
\le r/(1+c)$. Therefore, in the following we consider two cases:
$$
1 \le a \le (1+c)^{-\frac{d+1}{2d}} r, \quad \text{and} \quad
(1+c)^{-\frac{d+1}{2d}} r < a \le C  r.
$$
\subsubsection*{Case 1} If $1 \le a \le (1+c)^{-\frac{d+1}{2d}} r$, then 
$$
\delta = a^{\frac{2d}{d+1}} r^{-\frac{d-1}{d+1}} \implies
\delta \le r/(1+c).
$$
Thus, we may apply Lemma \ref{lem:new} to conclude
$$
\chi_{A,r - c \delta,\delta}(x) \le \chi_{A,r}(x) \le \chi_{A,r + c
\delta,\delta}(x).
$$
Therefore, by the definition of $N_{A,r,\delta}$
$$
N_{A,r-c\delta,\delta} \le \#\{n \in \mathbb{Z}^d \cap A(r\Omega) \} \le N_{A,r+c\delta,\delta}.
$$
Moreover, since $\delta \le r/(1+c)$ applying the bound derived above gives
$$
\left| N_{A,r+c\delta,\delta} -|\Omega| r^d \right|  \lesssim_{\Omega,\varphi}
a^\frac{2d}{d+1} r^{d - \frac{2d}{d+1}},
$$
and similarly,
$$
\left| N_{A,r-c\delta,\delta} -|\Omega| r^d \right|  \lesssim_{\Omega,\varphi}
a^\frac{2d}{d+1} r^{d - \frac{2d}{d+1}}.
$$
Therefore, we conclude
$$
\left| \#\{ n \in \mathbb{Z}^d \cap A (r \Omega) \} -  |\Omega| r^d \right|
\lesssim_{\Omega,\varphi} a^\frac{2d}{d+1} r^{d -
\frac{2d}{d+1}} .
$$
Note $\varphi$ can be fixed such that it only depends on $\Omega$ (more
specifically, dependent only on the dimension of $\Omega$), so the proof is
complete for this case.
\subsubsection*{Case 2} 
If $(1+c)^{-\frac{d+1}{2d}} r < a \le C  r$, then we define
$$
\tilde{c} := a (1 + c)^{\frac{d+1}{2d}} r^{-1}.
$$
Since the determinant of $A$ is equal to $1$, there exists $1 < k < d$ such that
$$
a_1 \le \ldots \le a_k \le 1 \le a_{k+1} \le \ldots \le a_d.
$$
Define
$$
\tilde{a}_j = \left\{ \begin{array}{cc} 
a_j \cdot \tilde{c}^{-1} & \text{for } 1 \le j \le k \\
a_j \cdot \tilde{c}^\frac{k}{d -k} & \text{for } k < j \le d
\end{array} \right. .
$$
Suppose $\tilde{A} = \diag(\tilde{a}_1,\tilde{a}_2,\ldots,\tilde{a}_d)$. Then
by construction
$$
\det(\tilde{A}) = 1 \quad \text{and} \quad \left\| \tilde{A}^{-1}
\right\|_\infty = (1+c)^{-\frac{d+1}{2d}} r.
$$
By the domain monotonicity of lattice point counting and the result from Case 1
$$
\# \{ n \in \mathbb{Z}^d \cap A (r \Omega) \} \le 
\# \{ n \in \mathbb{Z}^d \cap \tilde{A}( \tilde{c} r \Omega) \} =
\mathcal{O}\left( (\tilde{c} r)^d \right).
$$
Since $1 \le \tilde{c} \le c^{\frac{d+1}{2d}} C$ and the constants $c$ and $C$
only depend on $\Omega$ we conclude that
$$
\# \{ n \in \mathbb{Z}^d \cap A(r \Omega) \} = \mathcal{O} \left( r^d \right),
$$
where the implicit constant only depends on $\Omega$. This completes the proof
of Step \ref{step1}.

\end{proof}

\begin{step} \label{step2}
Suppose that $1 \le a \le C r$, and assume that $\Omega$ is balanced in the
sense of Definition \ref{balanced}. Let $\Omega_j$ denote the intersection of
$\Omega$ with the coordinate hyperplane orthogonal to the $j$-th coordinate
vector.  Then
$$
\# \left\{ n \in \mathbb{Z}^d \cap A  \left(r \bigcup_{j=1}^d \Omega_j \right)
\right\} = (\tr A^{-1}) r^{d-1} + \mathcal{O} \left(a^{\frac{2d}{d+1}}
r^{d-\frac{2d}{d+1}} \right),
$$
where the implicit constant is independent of $A$ and $r$.
\end{step}

\begin{proof}
If  $d=1$, then the statement is trivial. We consider two cases:
$d=2$ and $d > 2$.

\subsubsection*{Case 1} If the dimension $d=2$, then the set of positive
diagonal matrices of determinant $1$ is the $1$-parameter family $A =
\diag(1/a,a)$.  Therefore, it suffices to show that
$$
\# \left\{ n \in \mathbb{Z}^2 \cap A \left( r \bigcup_{j=1}^2 \Omega_j \right)
\right\} = \left( a + \frac{1}{a} \right) r + \mathcal{O} \left( a^{\frac{4}{3}}
r^{\frac{2}{3}} \right),
$$
for $A = \diag(1/a,a)$. In this case, $\Omega_1$ and $\Omega_2$ are the
intersection of $\Omega$ with the $y$-axis and $x$-axis, respectively. Moreover,
since we have assumed $\Omega$ is balanced $|\Omega_j| = 1$ for $j=1,2$.
Therefore, when $\Omega$ is scaled by $r$ and transformed by $A$, the total
number of points on the axes will be $a r + r/a + \mathcal{O}(1)$, and thus the
above statement holds.

\subsubsection*{Case 2} Suppose the dimension $d >2$. Let $A_j$ denote the
$(d-1) \times (d-1)$ positive diagonal matrix formed by removing the $j$-th row
and $j$-th column from $A$. Write:
$$
\# \left\{ n \in \mathbb{Z}^d \cap A_j(r\Omega_j) \right\} =
\# \left\{ n \in \mathbb{Z}^d \cap  a_j^{\frac{1}{d-1}} A_j \left(r
a_j^{-\frac{1}{d-1}} \Omega_j \right) \right\}.
$$
Observe that 
$$
\det a_j^{\frac{1}{d-1}} A_j = 1 \quad \text{and} \quad
\left\| \left(a_j^{\frac{1}{d-1}} A_j \right)^{-1} \right\|_\infty \le
\left( a_j^{-\frac{1}{d-1}} \right) a.
$$ 
Therefore, by the result from
Step \ref{step1}:
$$
\left| \# \left\{ n \in \mathbb{Z}^d \cap a_j^{\frac{1}{d-1}} A_j  \left(r
a_j^{-\frac{1}{d-1}} \Omega_j \right) \right\} - |\Omega_j| a_j^{-1} r^{d-1}
\right| \lesssim_\Omega 
\left(a_j^{-\frac{1}{d-1}} a \right)^\frac{2(d-1)}{d} 
\left( a_j^{-\frac{1}{d-1}} r \right)^{(d-1) - \frac{2(d-1)}{d}} .
$$
Observe that the total contribution of $a_j$ to the right hand side of this
inequality is
$$
a_j^{\left( -\frac{1}{d-1} \right) \left( \frac{2(d-1)}{d} + (d-1) -
\frac{2(d-1)}{d} \right)} = a_j^{-1}.
$$
Therefore, bounding $a_j^{-1}$ by $a = \|A^{-1}\|_\infty$ gives
$$
\left| \# \left\{ n \in \mathbb{Z}^d \cap a_j^{\frac{1}{d-1}} A_j  \left(r
a_j^{-\frac{1}{d-1}} \Omega_j \right) \right\} - |\Omega_j| a_j^{-1} r^{d-1}
\right| \lesssim_\Omega a \cdot \left( a^\frac{2(d-1)}{d}  r^{(d-1) -
\frac{2(d-1)}{d}} \right).
$$
A direct computation shows that
$$
\left( a^\frac{2d}{d+1} r^{d - \frac{2d}{d+1}} \right) \left( a^\frac{2(d-1)}{d}
r^{(d-1) - \frac{2(d-1)}{d}}\right)^{-1}  = a^\frac{2}{d^2+d} r^\frac{d^2 + d -
2}{d^2+d}.
$$
Thus, since we have assumed that $a \le C r$, it follows that
$$
a \cdot \left( a^\frac{2(d-1)}{d}  r^{(d-1) - \frac{2(d-1)}{d}} \right)
\lesssim_\Omega a^\frac{2d}{d+1} r^{d - \frac{2d}{d+1}}.
$$
Therefore, we obtain the bound
$$
\left| \# \{ n \in \mathbb{Z}^d \cap A_j(r\Omega_j) \} - |\Omega_j| a_j^{-1}
r^{d-1} \right| \lesssim_\Omega  a^\frac{2d}{d+1} r^{d -
\frac{2d}{d+1}}.
$$
By the balanced assumption $|\Omega_j| = 1$. Therefore, summing over
$j=1,\ldots,d$ yields
$$
\left| \sum_{j=1}^d \# \{ n \in \mathbb{Z}^d \cap A_j(r\Omega_j) \} - (\tr A^{-1})
r^{d-1} \right| \lesssim_\Omega a^\frac{2d}{d+1} r^{d - \frac{2d}{d+1}}.
$$
Next, we will use a similar argument to analyze $\mathbb{Z}^{d} \cap A_{j,k}
(r \Omega_{j,k})$, where $\Omega_{j,k} = \Omega_j \cap \Omega_k$ and $A_{j,k}$
denotes the $(d-2) \times (d-2)$ matrix formed by removing the $j$-th and $k$-th
rows, and $j$-th and $k$-th columns from $A$. Write:
$$
\# \left\{ n \in \mathbb{Z}^d \cap A_{j,k}(r\Omega_{j,k}) \right\} = \# \left\{
n \in \mathbb{Z}^d \cap (a_j a_k)^{\frac{1}{d-2}} A_{j,k} \left(r (a_j
a_k)^{-\frac{1}{d-2}} \Omega_j \right) \right\}.
$$
In this case, applying Step \ref{step1} yields:
$$
\left| \# \left\{ n \in \mathbb{Z}^d \cap (a_j a_k)^{\frac{1}{d-2}} A_{j,k}
\left(r (a_j a_k)^{-\frac{1}{d-2}} \Omega_j \right) \right\} - |\Omega_{j,k}|
a_j^{-1} a_k^{-1} r^{d-2} \right| 
$$
$$
\lesssim_\Omega \left( (a_j a_k)^{-\frac{1}{d-2}} a\right)^{\frac{2(d-2)}{d-1}}
\left( (a_j a_k)^{\frac{1}{d-2}} r\right)^{(d-2) - \frac{2(d-2)}{d-1}}.
$$
Observe that the total contribution of $a_j$ and $a_k$ to the right hand side is
$$
(a_j a_k)^{\left(-\frac{1}{d-2}\right) \left(\frac{2(d-2)}{d-1} + (d-2) -
\frac{2(d-2)}{d-1} \right)} = a_j^{-1} a_k^{-1}.
$$
Therefore, bounding $a_j^{-1} a_k^{-1}$ by $a^2$ yields the bound
$$
\left| \# \left\{ n \in \mathbb{Z}^d \cap A_{j,k}(r\Omega_{j,k}) \right\}  
- |\Omega_{j,k}|
a_j^{-1} a_k^{-1} r^{d-2} \right| 
\lesssim_\Omega a^2 \left( a^{\frac{2(d-2)}{d-1}}
 r^{(d-2) - \frac{2(d-2)}{d-1}} \right).
$$
A direct computation yields:
$$
\left( a^\frac{2d}{d+1} r^{d - \frac{2d}{d+1}} \right) \left(
a^\frac{2(d-2)}{d-1} r^{(d-2) - \frac{2(d-2)}{d-1}}\right)^{-1}  =
a^\frac{4}{d^2-1} r^\frac{2d^2 - 6}{d^2-1}.
$$
Therefore, since we have assumed $a \le C r$, it follows that
$$
a^2 \left( a^{\frac{2(d-2)}{d-1}} r^{(d-2) - \frac{2(d-2)}{d-1}} \right)
\lesssim_\Omega a^\frac{2d}{d+1} r^{d - \frac{2d}{d+1}} .
$$
Similarly, it follows that
$$
|\Omega_{j,k}| a_j^{-1} a_k^{-1} r^{d-2} \lesssim_\Omega a^\frac{2d}{d+1} r^{d -
\frac{2d}{d+1}}.
$$
Therefore, we conclude that
$$
\# \left\{ n \in \mathbb{Z}^d \cap A_{j,k}(r\Omega_{j,k}) \right\}
\lesssim_\Omega a^\frac{2d}{d+1} r^{d - \frac{2d}{d+1}} .
$$
Summing over $1 \le j < k \le d$ gives
$$
\sum_{1 \le j < k \le d}  \# \left\{ n \in \mathbb{Z}^d \cap
A_{j,k}(r\Omega_{j,k}) \right\}  \lesssim_\Omega a^\frac{2d}{d+1} r^{d -
\frac{2d}{d+1}}.
$$
Recall that we previously established the inequality
$$
\left| \sum_{j=1}^d \# \left\{ n \in \mathbb{Z}^d \cap A_j(r\Omega_j) \right\} -
(\tr A^{-1}) r^{d-1} \right| \lesssim_\Omega a^\frac{2d}{d+1} r^{d -
\frac{2d}{d+1}} .
$$
These last two inequalities can be used to deduce the result in combination with
the observation that:
$$
\sum_{j=1}^d \# \left\{ n \in \mathbb{Z}^d \cap A_j(r\Omega_j) \right\} -
\sum_{1 \le j < k \le d} \# \left\{ n \in \mathbb{Z}^d \cap
A_{j,k}(r\Omega_{j,k}) \right\} \le \# \left\{ n \in \mathbb{Z}^d \cap A
\left(r \bigcup_{j=1}^d \Omega_j \right) \right\},$$
and
$$
  \# \left\{ n \in \mathbb{Z}^d \cap A
\left(r \bigcup_{j=1}^d \Omega_j \right) \right\}\le \sum_{j=1}^d \# \{ n \in
\mathbb{Z}^d \cap A_j(r\Omega_j) \}.
$$
\end{proof}

\begin{step} \label{step3}
Assume that $\Omega$ is balanced in the sense of Definition \ref{balanced} and
is symmetric with respect to each coordinate hyperplane. Then
$$
\# \left\{ n \in \mathbb{Z}_{>0}^d \cap A(r \Omega) \right\} = 
\frac{1}{2^d} |\Omega| r^d - \frac{1}{2^d} \left( \tr A^{-1} \right) r^{d-1} + 
\mathcal{O} \left(a^{\frac{2d}{d+1}} r^{d - \frac{2d}{d+1}} \right).
$$
\end{step}
\begin{proof}
We partition the proof into two cases:
$$
1 \le a \le C r, \quad \text{and} \quad C r < a < \infty.
$$
\subsubsection*{Case 1} If $1 \le a \le C r$, then the results from Steps
\ref{step1} and \ref{step2} hold.  By the assumed symmetry of the domain, the
number of positive lattice points contained in $\Omega$ is equal to the number
of lattice points in $\Omega$ minus those in $\bigcup_{j=1}^d \Omega_j$
divided by $2^d$:
$$
\# \left\{ n \in \mathbb{Z}^d_{>0} \cap A(r \Omega) \right\} = \frac{1}{2^d}
\left( \# \left\{ n \in \mathbb{Z}^d \cap A(r \Omega) \right\} - \# \left\{ n
\in \mathbb{Z}^d \cap A \left(r \bigcup_{j=1}^d \Omega_j \right) \right\}
\right).
$$
Combing the results of Steps \ref{step1} and \ref{step2} yields the result.

\subsubsection*{Case 2} If $C r < a < \infty$, then we argue as follows. Recall
that $C > 0$ is a constant such that $\Omega \subset [-C,C]^d$. Therefore, if
$a > C r$, then $a_1 < 1/(C r)$ and hence
$$
A(r \Omega) \subseteq (-1,1) \times \mathbb{R}^{d-1}.
$$
In particular, it follows that
$$ 
\# \{ n \in \mathbb{Z}^d_{>0} \cap A(r \Omega) \} = 0.
$$
Therefore, the statement to prove reduces to
$$
0 = \frac{1}{2^d} |\Omega| r^d - \frac{1}{2^d} \left( \tr A^{-1} \right) r^{d-1}
+ \mathcal{O} \left(a^{\frac{2d}{d+1}} r^{d - \frac{2d}{d+1}} \right),
$$
which trivially holds because the error term dominates the right hand side since
$C r < a < \infty$.
\end{proof}

\begin{step} \label{step4}
Suppose that $\Omega$ is balanced in the sense of Definition \ref{balanced} and
is symmetric with respect to each coordinate hyperplane. Let
$$
A(r) = \argmax_A \, \# \left\{n \in \mathbb{Z}^d_{>0} \cap A(r\Omega) \right\},
$$
where the $\argmax$ ranges over all positive diagonal matrices $A$ of
determinant $1$.  Write $A(r) = \diag( a_1(r),a_2(r),\ldots,a_d(r))$.  Without
loss of generality suppose $a_1(r) \le a_2(r) \le \cdots \le a_d(r)$ and let
$a(r) = 1/a_1(r) = \|A^{-1}(r)\|_\infty$. Then
$$
\lim_{r\rightarrow \infty} a(r) = 1.
$$
\end{step}

\begin{remark} Two technical remarks are in order about the definition
of $A(r)$.  First, we argue why the $\argmax$ exists. As noted in Step
\ref{step3}, if $C r < a < \infty$, then $A(r\Omega)$ will not contain any
positive lattice points; therefore, the admissible $A$ for the $\argmax$ can be
restricted to those satisfying 
$$ 
1 \le \left\|A^{-1} \right\|_\infty \le C r \implies A(r \Omega) \subset
\left[-(Cr)^d, (C r)^2 \right]^d, 
$$ 
where $C > 0$ is a constant such that $\Omega \subset [-C,C]^d$.
Since the set $[-(Cr)^2,(Cr)^2]^d$ contains finitely many positive lattice
points, we conclude that the $\argmax$ exists.  Second, when the $\argmax$ is
not unique, $A(r)$ should be interpreted as being equal to an arbitrary element
from the maximal set; all convergence results are independent of this choice.
\end{remark}

\begin{proof}[\textbf{\bf Proof of Step \ref{step4}}]
By the argument in Case 2 of Step \ref{step3}, we may assume $1 \le a(r) \le C
r$.  First, we will suppose that there exists a constant $c > 0$ such that
$$
 \frac{1}{c} r \le a(r) \le C r,
$$
for arbitrarily large values $r$. That is to say, we suppose that there exists a
sequence $r_n$ tending towards infinity such that $r_n/c \le a(r_n) \le C r_n$.
We show that this assumption leads to a contradiction, which implies that
$$
a(r)/r \rightarrow 0, \quad \text{as } r \rightarrow \infty.
$$
While producing this contradiction, we write $A = A(r)$ and $a = a(r)$ to
simplify notation. Let $\chi$ denote the indicator function for $\Omega$. Then
$$
\# \left\{ n \in \mathbb{Z}^d_{>0} \cap A(r\Omega) \right\} = \sum_{n \in
\mathbb{Z}^d_{>0}} \chi(A^{-1} n/r ).
$$
Let $(n_1,n^\prime) \in \mathbb{Z}_{>0} \times \mathbb{Z}_{>0}^{d-1}$, and let
$A_1$ denote the $(d-1) \times (d-1)$ matrix formed by removing the $1$-st row
and $1$-st column from $A$.  With this notation
$$
\# \left\{ n \in \mathbb{Z}^d_{>0} \cap A(r\Omega) \right\} = 
\sum_{n_1 \in \mathbb{Z}_{>0}} \sum_{n^\prime \in \mathbb{Z}^{d-1}_{>0}} \chi
\left( \frac{n_1}{a_1 r},\frac{A_1^{-1} n^\prime}{r} \right).
$$
Since by assumption $r/c \le a \le C r$, and $a = 1/a_1$, we have
$$
\frac{1}{C} \le r a_1 \le c .
$$
For simplicity, assume $r a_1 = c$. In the course of the proof, we will only use
the fact that $ra_1$ is bounded above by a fixed constant. We remark that
the method of estimating the sum in this part of the proof is similar to the
arguments in \cite{GittinsLarson2017, vandenBergGittins2016}. With the notation
$r a_1 = c$
$$
\# \{ n \in \mathbb{Z}^d_{>0} \cap A(r\Omega) \} = 
\sum_{n_1 \in \mathbb{Z}_{>0}} \sum_{n^\prime \in \mathbb{Z}^{d-1}_{>0}} \chi
\left( \frac{n}{c},\frac{A_1^{-1} n^\prime}{r} \right).
$$
We assert that:
$$
\sum_{n^\prime \in \mathbb{Z}^{d-1}_{>0}} \chi
\left( \frac{n}{c},\frac{A_1^{-1} n^\prime}{r} \right)  \le
\int_{\mathbb{R}^{d-1}_{\ge 0}} \chi \left( \frac{n_1}{c}, \frac{A_1^{-1}
x^\prime}{r} \right) dx^\prime.
$$
Indeed, each lattice point $(n_1,n^\prime)$ can be identified with the
$(d-1)$-dimensional cube with
sides parallel to the coordinates axes and vertices $(n_1,n^\prime)$ and
$(n_1,n^\prime-\vec{1})$, where $\vec{1}$ denotes a $(d-1)$-dimensional vector
of ones. Since the set $A(r \Omega)$ is convex, each of
these cubes is contained in $A(r \Omega)$ and their first coordinate is equal to
$n_1$.  Moreover, this collection of cubes is disjoint since each pair of
vertices is unique. Therefore, the sum is a lower approximation of the integral.
Thus we have
$$
\# \left\{ n \in \mathbb{Z}^d_{>0} \cap A(r\Omega) \right\} 
\le
\sum_{n_1 \in \mathbb{Z}_{>0}} \int_{\mathbb{R}^{d-1}_{\ge 0}} \chi \left(
\frac{n_1}{c}, \frac{A_1^{-1} x^\prime}{r} \right) dx^\prime.
$$
Define the function $f_{A_1,r} : \mathbb{R} \rightarrow \mathbb{R}$ by 
$$
f_{A_1,r}(x_1) = \int_{\mathbb{R}^{d-1}_{\ge 0}} \chi \left(
\frac{x_1}{c}, \frac{A_1^{-1} x^\prime}{r} \right) dx^\prime.
$$
The function $f_{A_1,r}(x_1)$ is supported on some interval $[0,b]$ such that $0
< b < c  C$, where $C > 0$ is a constant such that $\Omega \subset [-C,C]^d$.
Define
$$ 
\Omega_{x_1/c} := \Omega \cap \left\{ \frac{x_1}{c} \right\} \times
\mathbb{R}^{d-1} \quad \text{such that} \quad f_{A_1,r}(x_1) = \frac{1}{2^{d-1}}
|A_1(r \Omega_{x_1/c})|,
$$ 
where $|\cdot|$ denotes the $(d-1)$-dimensional Lebesgue measure, and where the
factor $1/2^{d-1}$ arises since the integral is taken over
$\mathbb{R}^{d-1}_{\ge 0}$.  Since $\det(A_1) \cdot a_1 = 1$ and $a_1 = c/r$ it
follows that 
$$
f_{A_1,r}(x_1) =   \frac{1}{2^{d-1}} \frac{r}{c}|r \Omega_{x_1/c}| =
\frac{1}{2^{d-1}} r^d \frac{1}{c} |\Omega_{x_1/c}|.
$$
Since $\Omega$ is symmetric about the coordinate axes and is strictly convex,
the maximum value of $f_{A_1,r}(x_1)$ occurs at $x_1 = 0$, $f_{A_1,r}(x_1)$ is strictly
decreasing on $[0,b]$, and $f_{A_1,r}(b) = 0$.  Moreover, since we assume the
boundary of $\Omega$ is $C^{d+2}$ the function $f_{A_1,r}$ is certainly $C^1$,
which is sufficient for our purposes.  Integrating $f_{A_1,r}$ on $[0,b]$ yields
$$
\int_0^b f_{A_1,r}(x_1) dx_1 = \frac{1}{2^{d-1}}  r^d \int_0^b \frac{1}{c}
|\Omega_{x_1/c}| dx_1 = \frac{1}{2^d} |\Omega| r^d.
$$
However, our previous arguments show that
$$
\# \left\{ n \in \mathbb{Z}^d_{>0} \cap A(r\Omega) \right\} \le \sum_{n_1 \in
\mathbb{Z}_{>0}} f_{A_1,r}(x_1) = \frac{1}{2^{d-1}} r^d \sum_{n_1=1}^{\lfloor b
\rfloor} \frac{1}{c} |\Omega_{x_1/c}|.
$$
We assert that there exists $\varepsilon > 0$ such that
$$
\sum_{n_1=1}^{\lfloor b \rfloor} \frac{1}{c} |\Omega_{n_1/c}| \le
(1-\varepsilon) \int_0^b \frac{1}{c} |\Omega_{x_1/c}| dx_1.
$$
Indeed, the sum on the left hand side is a lower Riemann sum for the integral of
the strictly decreasing $C^1$ function $(1/c)|\Omega_{x_1/c}|$. Moreover, the
constant $\varepsilon > 0$ can be chosen to hold uniformly for any lower Riemann
sum of the integral which discretizes the integral into at most $\lfloor c C
\rfloor$ pieces. 
\begin{figure}[h!]
\centering
\includegraphics[width=.25\textwidth]{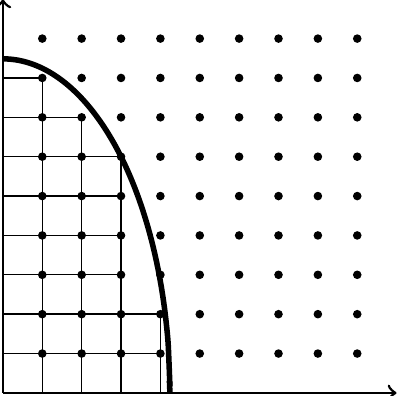}
\caption{An approximation of the area enclosed by an ellipse by squares
$[i_1-1,i_1] \times [i_2-1,i_2]$.}
\end{figure}

Therefore, we conclude that
$$
\# \{n \in \mathbb{Z}^d_{>0} \cap A(r\Omega) \} \le \sum_{n_1=1}^{\lfloor b
\rfloor} f_{A_1,r}(n_1) \le (1
- \varepsilon) \int_0^b f_{A_1,r}(x_1) dx_1 = (1-\varepsilon) \frac{1}{2^d} |\Omega|
r^d.
$$
However, applying Step \ref{step3} with $A = \Id$ yields
$$
\# \{ n \in \mathbb{Z}_{>0}^d \cap \Id(r \Omega) \} = 
\frac{1}{2^d} |\Omega| r^d - \frac{1}{2^d} d r^{d-1} + 
\mathcal{O} \left( r^{d - \frac{2d}{d+1}} \right).
$$
When $r$ is sufficiently large
$$
-\varepsilon r^d <
- \frac{1}{2^d} d r^{d-1} + \mathcal{O} \left( r^{d - \frac{2d}{d+1}} \right),
$$
which contradicts the optimality of such $a \ge r/c$. Therefore, we conclude
that 
$$
a(r) = \psi(r) r, \quad \text{where} \quad \psi(r) \rightarrow 0, \quad
\text{as} \quad r \rightarrow \infty.
$$ 
By the Step \ref{step3} of the proof:
$$
\#\{n \in \mathbb{Z}_{>0}^d \cap A(r)(r\Omega) \} = \frac{1}{2^d} |\Omega| r^d -
\frac{1}{2^{d}} \left( \tr A(r)^{-1} \right) r^{d-1} + \mathcal{O} \left(
a(r)^\frac{2d}{d+1} r^{d - \frac{2d}{d+1}} \right).
$$
Substituting $a(r) = \psi(r) r$ yields
$$
\# \left\{ n \in \mathbb{Z}_{>0}^d \cap A(r)(r\Omega) \right\} \le \frac{1}{2^d}
|\Omega| r^d - \left(\frac{1}{2^{d}} - C_1 \psi(r)^\frac{d-1}{d+1} \right)
\psi(r) r^d,
$$
for some constant $C_1$.  Since $\psi(r) \rightarrow 0$, we may choose $r$ large
enough such that 
$$
\frac{1}{2^{d}} - C_1 \psi(r)^\frac{d-1}{d+1} \ge \frac{1}{2^{d+1}}.
$$
Thus for large enough $r$,
$$
\# \left\{n \in \mathbb{Z}_{>0}^d \cap A(r) (r \Omega) \right\} \le
\frac{1}{2^{d}} |\Omega| r^d - \frac{1}{2^{d+1}}\psi(r) r^d.  
$$
However, if such a situation is optimal, it must be competitive with the
situation $A = \Id$, where
$$
\# \left\{ n \in \mathbb{Z}_{>0}^d \cap \Id (r\Omega) \right\} = \frac{1}{2^d}
|\Omega| r^d - \frac{d}{2^{d}} r^{d-1} + \mathcal{O} \left( r^{d -
\frac{2d}{d+1}} \right).
$$
Therefore, we conclude that
$$
\psi(r) r^d  = \mathcal{O} \left(r^{d-1} \right),
$$
which implies
$$
\psi(r) = \mathcal{O} \left(r^{-1} \right).
$$
Therefore, the set of optimal $a$ is uniformly bounded. Convergence then
immediately follows from the result from Step \ref{step3}, i.e., from the equation:
$$
\# \left\{ n \in \mathbb{Z}_{>0}^d \cap A(r \Omega) \right\} = 
\frac{1}{2^d} |\Omega| r^d - \frac{1}{2^d} \left( \tr A^{-1} \right) r^{d-1} + 
\mathcal{O} \left(a^{\frac{2d}{d+1}} r^{d - \frac{2d}{d+1}} \right).
$$
Indeed, since $a$ is uniformly bounded the second term determines the effect of
$A$ when $r$ is large. Therefore, in order to maximize $\# \{ n \in
\mathbb{Z}_{>0}^d \cap A(r \Omega) \}$ the coefficient of $r^{d-1}$ must be
minimized. More precisely, $A(r)$ can be characterized in terms of the following
minimization problem
$$
A(r)= \argmin_A \left( \tr A^{-1} \right) + \mathcal{O}\left(r^{-\frac{d -
1}{d+1}} \right)  
$$
where the $\argmin$ is taken over positive diagonal matrices $A$ of determinant
$1$ such that $1 \le \|A^{-1}\|_\infty \le c_0$, where $c_0 \ge 1$ is a fixed
constant. By the arithmetic mean geometric mean inequality, any positive
diagonal matrix $A$ of determinant $1$ satisfies
$$
\tr A^{-1} \ge d \left( \det A^{-1} \right)^{1/d} = d,
$$
with equality if and only if $A = \Id$.  Therefore, we conclude that
$$
\tr A(r)^{-1}  \rightarrow d, \quad \text{as } r \rightarrow
\infty.  
$$
Since $\tr A^{-1}$ is a continuous function, and equality holds in
the arithmetic mean geometric mean inequality if and only if $A = \Id$ we
conclude
$$
\|A(r) - \Id \|_\infty \rightarrow 0, \quad \text{as } r \rightarrow \infty,
$$
as was to be shown.
\end{proof}

In the fifth step, we establish a rate of convergence.

\begin{step} \label{step5}
Suppose that $\Omega$ is balanced in the sense of Definition \ref{balanced} and
is symmetric with respect to each coordinate hyperplane. Let
$$
A(r) = \argmax_A \, \#\{n \in \mathbb{Z}^d_{>0} \cap A(r\Omega) \},
$$
where the $\argmax$ ranges over all positive diagonal matrices $A$ of determinant $1$. Then
$$
\|A(r) - \Id\|_\infty = \mathcal{O}\left( r^{-\frac{d-1}{2(d+1)}} \right).
$$
\end{step}
\begin{proof}
Applying the result of Step \ref{step3} for $A = \Id$ gives
$$
\# \{ n \in \mathbb{Z}^d_{>0} \cap \Id( r \Omega) \} = \frac{1}{2^d} |\Omega|
r^d - \frac{1}{2^d} d r^{d-1} + \mathcal{O}(r^{d-\frac{2d}{d+1}}).
$$
By the arithmetic mean geometric mean inequality
$$
\tr A^{-1} \ge d \left( \det A^{-1} \right)^{1/d}  = d.
$$
Therefore, since $A(r)$ maximizes $\# \{ n \in \mathbb{Z}^d_{>0} \cap A (r
\Omega) \}$ over all $A$, we must have
$$
\frac{1}{2^d} \tr A^{-1}(r) r^{d-1} - \frac{1}{2^d} d r^{d-1}
\lesssim_\Omega r^{d - \frac{2d}{d+1}},
$$
which simplifies to
$$
\tr A^{-1}(r) -d \lesssim_\Omega r^{-\frac{d-1}{d+1}}.
$$
To complete the proof it suffices to show that:
$$
\frac{1}{8} \|A-\Id\|^2_\infty \le \tr A^{-1} -d.
$$
Indeed then,
$$
\frac{1}{8} \|A-\Id\|^2_\infty \le \tr A^{-1} -d  \lesssim_\Omega r^{-\frac{d-1}{d+1}}
\quad \implies \quad \|A-\Id\|_\infty \lesssim_\Omega r^{-\frac{d-1}{2(d+1)}}.
$$
Without loss of generality, suppose that 
$$
A(r) = \diag \left(a_1,a_2,\ldots,a_d \right) \quad \text{where} \quad a_1 \le
a_2 \le \ldots \le a_d.
$$
There are two cases to consider
\subsubsection*{Case 1} $\|A(r) - \Id\|_\infty = |a_d - 1|$. In this case,
suppose that $a_d = 1 + \varepsilon$ where $\varepsilon > 0$. Then
$$
\|A(r) - \Id \|_\infty = \varepsilon.
$$
Since the determinant of $A$ is equal $1$, it follows that the product $ a_1
\cdot a_2 \cdot \cdots \cdot a_{d-1} = \frac{1}{1+\varepsilon}$.  Therefore, by
the arithmetic mean geometric mean inequality,
$$
(d-1) ( 1 + \varepsilon )^\frac{1}{d-1} + \left( \frac{1}{1+\varepsilon} \right)
\le \frac{1}{a_1} + \frac{1}{a_2} + \cdots \frac{1}{a_{d-1}} + \left(
\frac{1}{1+\varepsilon} \right) = \tr A^{-1}.
$$
Expanding the left hand side in a Taylor series yields
$$
d + \frac{d \varepsilon^2}{2(d-1)} + \mathcal{O}(\varepsilon^3) = \tr A^{-1}.
$$
And therefore, when $\varepsilon$ is sufficiently small,
$$
\frac{1}{2} \|A(r) - \Id\|_\infty^2  =
\frac{1}{2} \varepsilon^2 \le \tr A^{-1} - d .
$$
\subsubsection*{Case 2} $\|A(r) - \Id\|_\infty = |a_1 - 1|$. Suppose that $a_1 = \frac{1}{1+\varepsilon}$. When $\varepsilon$ is
sufficiently small
$$
\|A(r) - \Id \|_\infty = |a_1 - 1 | \le 2 \varepsilon.
$$
Furthermore, since the determinant of $A$ is equal to $1$
$$
a_2 \cdot a_3 \cdot \cdots \cdot a_d = 1+\varepsilon.
$$ 
Therefore, by the arithmetic mean geometric mean inequality
$$
(1 + \varepsilon) + (d-1) \left(\frac{1}{1+\varepsilon} \right)^\frac{1}{d-1}
\le (1 + \varepsilon) + \frac{1}{a_2} + \frac{1}{a_3} + \cdots \frac{1}{a_d} =
\tr A^{-1}.
$$
Expanding the left hand side in a Taylor series yields
$$
d + \frac{d \varepsilon^2}{2(d-1)} + \mathcal{O}(\varepsilon^3) \le \tr A^{-1}.
$$
Therefore, when $\varepsilon$ is sufficiently small,
$$
\frac{1}{8} \|A - \Id|^2 \le \frac{1}{2} \varepsilon^2 \le \tr A^{-1} - d.
$$
Moreover, in either case
$$
\frac{1}{8} \|A-\Id\|^2 \le \tr A^{-1} -d,
$$
and the proof is complete.
\end{proof}

\textbf{Acknowledgment.} We are grateful to Stefan Steinerberger for many useful
discussions, and Andrei Deneanu for insightful comments. Additionally, we would
like to thank Richard Laugesen for valuable feedback leading to corrections in
the proof and a much improved exposition, and the referees for their helpful
comments.

\end{document}